\documentclass{article}
\usepackage[margin =1in]{geometry}

\usepackage{amssymb,amsthm,amsmath}
\usepackage{enumerate}
\usepackage{enumitem}
\usepackage[hidelinks]{hyperref}
\usepackage{cite}
\usepackage{cleveref}
\usepackage{color}
\usepackage[dvipsnames]{xcolor}
\usepackage{fancyvrb}
\usepackage{listings}
\usepackage{graphicx}
\usepackage{booktabs}
\usepackage{subcaption}

\usepackage{amsmath, amsthm, amssymb}
\usepackage{mathtools}
\usepackage[parfill]{parskip}
\usepackage{soul}
\usepackage{multirow}
\usepackage{caption}
\usepackage{subcaption}

\usepackage{epsfig,psfrag,url,verbatim}
\usepackage{graphics}
\usepackage{graphicx}
\usepackage{epstopdf}

\usepackage{amssymb}
\usepackage{amsthm}
\usepackage{url}
\usepackage[ruled,linesnumbered]{algorithm2e}
\usepackage{enumitem}
\usepackage{booktabs}
\usepackage{mathtools}
\usepackage{siunitx}

\newcommand{\R}{{\mathbb R}}

\newcommand{\N}{{\mathbb N}}

\newcommand{\proj}{\mathop{\bf proj}}

\newcommand{\maximize}{\mathop{\rm maximize}}
\newcommand{\minimize}{\mathop{\rm minimize}}

\newcommand{\tx}{\hat{x}}

\newcommand{\bz}{\bar{z}}

\usepackage{nicefrac}

\usepackage{physics}

\newtheorem*{theorem*}{Theorem}

\newtheorem{proposition}{Proposition}

\usepackage{todonotes}

\SetKwFunction{GenericRestartScheme}{GenericRestartScheme}

\SetKwFunction{GenericFixedFrequencyRestartScheme}{GenericFixedFrequencyRestartScheme}

\SetKwFunction{GenericIterativeAlgorithm}{GenericIterativeAlgorithm}

\newcommand{\callGenericIterativeAlgorithm}[1]{\hyperref[function:GenericIterativeAlgorithm]{\GenericIterativeAlgorithm{#1}}}

\SetKwFunction{OneStepOfAlgorithm}{OneStepOfAlgorithm}
\SetKwFunction{InitializeAlgorithm}{InitializeAlgorithm}

\SetKwFunction{OneStepOfPDHG}{OneStepOfPDHG}
\SetKwFunction{InitializePDHG}{InitializePDHG}
\SetKwFunction{AdaptiveRestartPDHG}{AdaptiveRestartPDHG}

\SetKwFunction{OneStepOfAGD}{OneStepOfAGD}
\SetKwFunction{InitializeAGD}{InitializeAGD}
\SetKwFunction{AdaptiveRestartAGD}{AdaptiveRestartAGD}

\SetKwFunction{OneStepOfExtragradient}{OneStepOfExtragradient}
\SetKwFunction{InitializeExtragradient}{InitializeExtragradient}
\SetKwFunction{AdaptiveRestartExtragradient}{AdaptiveRestartExtragradient}

\SetKwFunction{PrimalWeightUpdate}{PrimalWeightUpdate}
\SetKwFunction{AdaptiveStepOfPDHG}{AdaptiveStepOfPDHG}
\SetKwFunction{GetRestartCandidate}{GetRestartCandidate}
\SetKwFunction{InitializePrimalWeight}{InitializePrimalWeight}
\newcommand{\epsMach}{\epsilon_{\text{mach}}}
\newcommand{\epsZero}{\epsilon_{\text{zero}}}

\SetKwProg{Fn}{Function}{:}{}

\newcommand{\InitialTau}[1]{\tau_{\text{init}}}

\newcommand{\xStar}{x^{\star}}
\newcommand{\yStar}{y^{\star}}

\newcommand{\Lag}[0]{\mathcal{L}}
\newcommand{\T}{\top}

\newcommand{\tA}{\tilde{A}}
\newcommand{\tK}{\tilde{K}}
\newcommand{\tG}{\tilde{G}}
\newcommand{\thh}{\tilde{h}}
\newcommand{\tb}{\tilde{b}}
\newcommand{\tu}{\tilde{u}}
\newcommand{\tl}{\tilde{l}}
\newcommand{\tc}{\tilde{c}}
\begin{document}
\lstset{
breaklines=true,
basicstyle=\small\ttfamily,
breakatwhitespace=true,
breakautoindent=false,
breakindent=0pt,
resetmargins=true,
}
	\title{Practical Large-Scale Linear Programming using\\ Primal-Dual Hybrid Gradient}

	\author{ David Applegate\thanks{
  Google Research (\texttt{\{dapplegate, mlubin, wschudy\}@google.com});} \qquad Mateo D\'iaz\thanks{California Institute of Technology
   (\texttt{mateodd@caltech.edu});} \qquad
  Oliver Hinder\thanks{
  Google Research and University of Pittsburgh (\texttt{ohinder@pitt.edu});} \qquad 
  Haihao Lu\thanks{
  Google Research and University of Chicago (\texttt{haihao.lu@chicagobooth.edu});} \\
  Miles Lubin$^\ast$  \qquad
  Brendan O'Donoghue\thanks{
  Deepmind (\texttt{bodonoghue@deepmind.com});}  \qquad
  Warren Schudy$^\ast$
}
\date{}
\maketitle
\begin{abstract}
We present PDLP, a practical first-order method for linear programming (LP) that can solve to the high levels of accuracy that are expected in traditional LP applications. In addition, it can scale to very large problems because its core operation is matrix-vector multiplications.
PDLP is derived by applying the primal-dual hybrid gradient (PDHG) method, popularized by Chambolle and Pock (2011), to a saddle-point formulation of LP. PDLP enhances PDHG for LP by combining several new techniques with older tricks from the literature; the enhancements include diagonal preconditioning, presolving, adaptive step sizes, and adaptive restarting. PDLP improves the state of the art for first-order methods applied to LP. We compare PDLP with SCS, an ADMM-based solver, on a set of 383 LP instances derived from MIPLIB 2017. With a target of $10^{-8}$ relative accuracy and 1 hour time limit, PDLP achieves a 6.3x reduction in the geometric mean of solve times and a 4.6x reduction in the number of instances unsolved (from 227 to 49).
Furthermore, we highlight standard benchmark instances and a large-scale application (PageRank) where our open-source prototype of PDLP, written in Julia, outperforms a commercial LP solver.
\end{abstract}
\section{Introduction}

First-order methods (FOMs), which use gradient and not Hessian information, are now applied as standard practice in many areas of optimization~\cite{BeckBook}. A known weakness of FOMs is the \textit{tailing-off} effect, where FOMs quickly find moderately accurate solutions, but progress towards an optimal solution slows down over time. While moderately accurate solutions are often sufficient for large machine learning applications, other applications traditionally demand higher precision. One such area is Linear Programming (LP), the focus of this work.

LP is a fundamental class of optimization problems in applied mathematics, operations research, and computer science with a huge range of applications, including mixed-integer programming, scheduling, network flow, chip design, budget allocation, and many others \cite{dantzig2016linear, schrijver1998theory, vanderbei2015linear, boyd2004convex}. Software for solving LP problems, called \textit{LP solvers}, originated in the earliest days of computing, predating the invention of operating systems~\cite{o-h1984history}. The state-of-the-art methods for LP, namely Dantzig's simplex method \cite{dantzig1990origins,dantzig2016linear} and interior-point (or barrier) methods \cite{nesterov1994interior}, are quite mature and reliable at delivering highly accurate solutions. These widely successful methods have left little room for FOMs to make inroads. Furthermore, practitioners who use LP solvers are not accustomed to reasoning about the trade-off between accuracy and computing times typically intrinsic to FOMs.

In this paper, we provide evidence that, if properly enhanced, FOMs can obtain high quality solutions to LP problems quickly. Indeed, there's reason to expect this, as authors have developed FOMs for LP with linear rates of convergence~\cite{wang2017admmlp, eckstein1990alternating, gilpin2012first, yang2018rsg, necoara2019linearfom}.
On the other hand, the linear rates depend on potentially loose and hard-to-compute constants; hence, tailing off may still be observed in practice.
To our knowledge, ours is the first work to combine both theoretical enhancements with practical heuristics, demonstrating their combined effectiveness with extensive computational experiments on standard benchmark instances. In fact, our experiments will expose a substantial gap between algorithms presented in the literature and what's needed to obtain good performance.

Starting from a baseline primal-dual hybrid gradient (PDHG) method~\cite{chambolle2011first} applied to a saddle point formulation of LP, we develop a series of algorithmic improvements. These enhancements include adaptive restarting~\cite{applegate2021restarts}, dynamic primal-dual step size selection \cite{goldstein2013adaptive, goldstein2015adaptive}, presolving techniques  \cite{achterberg2020presolve}, and diagonal preconditioning (data equilibration) \cite{giselsson2016linear}. Most of these enhancements, while inspired by existing literature, are novel.
We name our collection of enhancements \textit{PDLP} (PDHG for LP). 

The impact of these improvements is substantial. For example, on 383 LP instances derived from the MIPLIB 2017 collection~\cite{gleixner2021miplib}, our implementation of a baseline version of PDHG solved only 50 problems to $10^{-8}$ relative accuracy given a limit of approximately 100,000 iterations per problem. By contrast, PDLP solves 283 of the 383 problems under the same conditions. We demonstrate that PDLP outperforms FOM baselines and, in a small number of cases, obtains performance competitive with a commercial LP solver.

Although not the focus of this paper, we believe that our results open the door to a new set of possibilities and computational trade-offs when solving LP problems. PDLP has the potential to solve extremely large scale instances where the simplex method and interior-point methods are unable to run because of their reliance on matrix factorization. Since PDLP uses matrix-vector operations at its core, it can effectively run on multi-threaded CPUs, GPUs~\cite{SpmvGPUBenchmark2020}, or distributed clusters~\cite{EcksteinMatyasfalvi2018}. Furthermore, a GPU implementation of PDLP could efficiently solve batches of similar problems, a setup that has already been successfully applied with other optimization algorithms in applications like strong branching~\cite{nair2020neuralmip} and training neural networks that contain optimization layers~\cite{amos2017optnet}.

\paragraph{Outline.} The remainder of this section focuses on related work. Section \ref{sec:pdhg} introduces LP and PDHG. Section \ref{sec:practical-algorithmic-improvements} describes the set of enhancements that define PDLP. Section \ref{sec:experiments} presents numerical experiments, and Section \ref{sec:futurework} concludes and outlines future directions.

\subsection{Literature review}

\paragraph{PDHG} 
PDHG was first developed by Zhu and Chan \cite{zhu2008efficient}, with subsequent analysis and extension by a number of authors \cite{esser2010general, pock2009algorithm, chambolle2011first, he2012convergence, condat2013primal, chambolle2018stochastic,alacaoglu2019convergence}. PDHG is closely related to the Arrow-Hurwicz method \cite{arrow58}. PDHG is a form of operator-splitting \cite{bauschke2011convex, ryu2016primer} and can be interpreted as a variant of the alternating directions method of multipliers (ADMM) and Douglas-Rachford splitting (DRS) \cite{eckstein1992douglas, boyd2011distributed, o2020equivalence}, which themselves are both instantiations of the proximal point method \cite{rockafellar1976monotone, eckstein1992douglas, parikh2014proximal}. As opposed to ADMM or DRS, PDHG is `matrix-free' in that the data matrix is only used for matrix-vector multiplications. This allows PDHG to scale to problems even larger than those tackled by these other techniques, and to make better use of parallel and distributed computation.

\paragraph{FOM-based solvers} Recent interest in large-scale cone programming has sparked the development several first-order solvers based on competing methods. ProxSDP \cite{souto2020exploiting} is a solver for semidefinite programming based on PDHG. Solvers based on Nesterov's accelerated gradients \cite{nesterov1983method} include TFOCS \cite{becker2011templates}, and FOM which is a suite of solvers employing both gradient and proximal algorithms \cite{beck2019fom}.
Solvers based on operator splitting techniques like ADMM include SCS \cite{ocpb:16, scs, o2020operator}, OSQP \cite{stellato2018osqp}, POGS \cite{fougner2018parameter}, and COSMO \cite{garstka_2019}. Of these both SCS and POGS offer a matrix-free implementation where the linear system, that arises from the proximal operator used in ADMM, is solved using the conjugate gradient method. However, we shall show experimentally that our method can be significantly faster and more robust than this approach. Finally, \cite{ali2017semismooth} considers applying a truncated semismooth Newton method to the system of equations defining a fixed point of the SCS operator. 

\paragraph{FOMs for LP}
Lan, Lu and Monteiro \cite{Lan2011} and Renegar \cite{Renegar2019} develop FOMs for LP as a special case of semidefinite programming, with sublinear convergence rates. The FOM-based solvers above all apply to more general problem classes like cone programming or quadratic programming. In contrast, some of the enhancements that constitute PDLP are specialized, either in theory or practice, for LP (namely restarts~\cite{applegate2021restarts} and presolving). A number of authors \cite{wang2017admmlp, eckstein1990alternating, gilpin2012first, yang2018rsg, necoara2019linearfom} have proposed linearly convergent FOMs for LP; to our knowledge, none have been subject of a comprehensive computational study.  ECLIPSE~\cite{pmlr-v119-basu20a} solves huge-scale industrial LP problems by accelerated gradient descent, without presenting comparisons on standard test problems. Lin et al. \cite{lin2021admm} propose an ADMM-based interior point method. In contrast with PDLP which solves to high accuracy (i.e., $10^{-8}$ relative error), \cite{lin2021admm} perform experiments with $10^{-3}$ and $10^{-5}$ relative error. SNIPAL \cite{li2020asymptotically} is a semismooth Newton method based on the proximal augmented Lagrangian. SNIPAL has fast asymptotic convergence, yet, to get good performance, the authors use ADMM for warm-starts. Given PDLP’s favorable comparisons with SCS, it’s plausible that PDLP could provide a more effective warm-start.  Finally, Pock and Chambolle~\cite{pock2011diagonal} apply PDHG with diagonal preconditioning to a limited set of test LP problems and Applegate et al.~\cite{applegate2021infeasibility} show how to extract infeasibility certificates when applying PDHG to LP.

\section{Preliminaries}\label{sec:pdhg}

In this section, we introduce the notation we use throughout the paper,  summarize the LP formulations we solve, and introduce the baseline PDHG algorithm. 

\paragraph{Notation.}
Let $\R$ denote the set of real numbers, $\R^{+}$ the set of nonnegative real numbers, and $\R^{-}$ the set of nonpositive real numbers. Let $\N$ denote the set of natural numbers (starting from one). 
Let $\| \cdot \|_p$ denote the $\ell_p$ norm for a vector, and let $\|\cdot\|_2$ denote the spectral norm for a matrix. 
For a vector $v \in \R^n$, we use $v^+$ and $v^-$ for their positive and negative parts, i.e., $v_i^+ = \max\{0, v_i\}$ and $v_i^- = \min\{0, v_i\}$. The symbol $v_{1:m}$ denotes the vector with the first $m$ components of $v$. 
The symbols $K_{i,\cdot}$ and $K_{\cdot,j}$ correspond to the $i$th column and $j$th row of the matrix $K$, respectively.
The symbol $\mathbf{1}$ denotes the vector of all ones. Given a convex set $X$, we use $\proj_X$ to denote the map that projects onto $X.$

\paragraph{Linear Programming.} We solve primal-dual LP problems of the form:
\begin{equation}\label{eq:lp}
    \begin{aligned}[c]
    \minimize_{x\in \R^n}~~ &~ c^\top x \\
\text{subject to:}~~ &~ Gx \geq h \\
& ~ Ax = b \\
& ~ l \leq x \leq u
    \end{aligned}
    \qquad\qquad\qquad
    \begin{aligned}[c]
\maximize_{y\in \R^{m_1 + m_2}, \lambda\in\R^n} \quad & q^\T y + l^\T \lambda^{+} - u^\T \lambda^{-} \\
\text{subject to:}\quad & c - K^\T y = \lambda \\
&  y_{1:m_1} \geq 0 \\
& \lambda \in \Lambda\ ,
    \end{aligned}
\end{equation} 
where $G \in \R^{m_1\times n}$, $A \in \R^{m_2 \times n}$, $c \in \R^{n}$, $h \in \R^{m_1}$, $b \in \R^{m_2}$, $l \in (\R \cup \{ -\infty \})^{n}$, $u \in (\R \cup \{ \infty \})^{n}$, $K^\top = \begin{pmatrix} G^\T, A^\T \end{pmatrix}$, $q^\T := \begin{pmatrix}
h^\T,
b^\T
\end{pmatrix}$, and
$$
\Lambda = \Lambda_{1} \times \dots \times \Lambda_{n} \quad \Lambda_i :=
\begin{cases}
\{ 0 \} & l_i = -\infty, ~ u_i = \infty,  \\
\R^{-} & l_i = -\infty, ~ u_i \in \R \\
\R^{+} & l_i \in \R, ~ u_i = \infty \\
\R & \text{otherwise} 
\end{cases}
$$
is the set of variables $\lambda$ such that the dual objective is finite. This pair of primal-dual problems is equivalent to the  saddle-point problem:
\begin{flalign}\label{eq:primal-dual}
\min_{x \in X} \max_{y \in Y} \Lag(x,y) := c^\T x - y^\T K x + q^\T y
\end{flalign}
with $X := \{x \in \R^n : l \leq x \leq u \}$, and $Y := \{y \in \R^{m_1+m_2} : y_{1:m_1} \geq 0\}.$

\paragraph{PDHG.} When specialized to \eqref{eq:primal-dual}, the PDHG algorithm takes the form:
\begin{align}\label{eq:standard-pdhg}
\begin{split}
x^{k+1} &= \proj_{X}(x^k - \tau (c - K^\top y^k))\\
y^{k+1} &= \proj_{Y}(y^k + \sigma (q - K (2 x^{k+1} - x^k)))
\end{split}
\end{align}
where $\tau, \sigma > 0$ are primal and dual step sizes, respectively.  PDHG is known to converge to an optimal solution when $\tau \sigma \| K \|_2^2 \le 1$ \cite{condat2013primal,chambolle2016ergodic}.
We reparameterize the step sizes by \begin{equation}
    \tau = \eta / \omega  \quad \text{and} \quad \sigma = \omega \eta \qquad \text{with }\eta \in (0, \infty) \quad \text{and} \quad \omega \in (0,\infty).
\end{equation}
We call $\omega \in (0, \infty)$ the \emph{primal weight}, and $\eta \in (0,\infty)$ the \emph{step size}. Under this reparameterization PDHG converges for all $\eta \le 1 / \| K \|_2$. This allows us to control the scaling between the primal and dual iterates with a single parameter $\omega$. We use the term \emph{primal weight} to describe $\omega$ because it weights the primal variables in the following norm:
$$
\| z \|_{\omega} := \sqrt{ \omega \| x \|_2^2 + \frac{\| y \|_2^2}{\omega} }.
$$
This norm plays a role in the theory for PDHG \cite{chambolle2016ergodic} and later algorithmic discussions.

For the \emph{baseline PDHG algorithm} that we use for comparisons, we consider two simple choices for $\eta$ and $\omega$. For the step size, set $\eta = 0.9 / \| K \|_2$ where $\| K \|_2$ is estimated via power iteration, and for the primal weight we set $\omega = 1$; this is similar to the default parameters in the standard PDHG implementation in ODL~\cite{adler2017odl}.

\section{Practical algorithmic improvements}\label{sec:practical-algorithmic-improvements}

 In this section, we detail these enhancements, and defer further experimental testing of them to Section \ref{sec:experiments} and ablation studies to Appendix~\ref{sec:ablation}.
While our enhancements are inspired by theory, our focus is on practical performance. The algorithm as a whole has no convergence guarantee, although some individual enhancements do; see Section~\ref{sec:guarantees} for further discussion.

\newcommand{\zc}[0]{z_{\text{c}}}
\newcommand{\rc}[0]{r_{\text{c}}}

Algorithm~\ref{alg:practical-algorithm} presents pseudo-code for PDLP after preprocessing steps.
We modify the step sizes (Section~\ref{sec:step-size-choice-body}), add restarts (Section~\ref{sec:adaptive-restarts}), and dynamically update the primal weights (Section~\ref{sec:primal-weight-update}). Before running Algorithm~\ref{alg:practical-algorithm} we apply presolve (Section~\ref{sec:presolve}) and diagonal preconditioning (Section~\ref{sec:diag}).
There are some minor differences between the pseudo-code and the actual code.
In particular, we only evaluate the restart or termination criteria (Line~\ref{line:restart-or-termination}) every $40$ iterations. This reduces the associated overheads with minimal impact on the total number of iterations. We also
check the termination criteria before beginning the algorithm or if we detect a numerical error.

\begin{algorithm}
\SetAlgoLined
{\bf Input:} An initial solution $z^{0,0}$\;
Initialize outer loop counter $n\gets 0$, total iterations $k \gets 0$, step size $\hat \eta^{0,0} \gets 1/\| K \|_{\infty}$,  primal weight $\omega^{0} \gets$ \InitializePrimalWeight{c, q}\;
 \Repeat{termination criteria holds}{
 $t \gets 0$\;
 \Repeat{restart or termination criteria holds}{
    $z^{n,t+1}, \eta^{n,t+1}, \hat \eta^{n,t+1} \gets$ \AdaptiveStepOfPDHG{$z^{n,t}, \omega^{n}, \hat \eta^{n,t}, k$} \;
$\bz^{n,t+1}\gets\frac{1}{\sum_{i=1}^{t+1} \eta^{n,i}} \sum_{i=1}^{t+1} \eta^{n,i} z^{n,i}$ \; 
    $\zc^{n,t+1} \gets$ \GetRestartCandidate{$z^{n,t+1}, \bz^{n,t+1}, z^{n,0}$} \;
    $t\gets t+1$, $k \gets k + 1$ \;
 }\label{line:restart-or-termination}
  \textbf{restart the outer loop.} $z^{n+1,0}\gets \zc^{n,t}$, $n\gets n+1$\;
  $\omega^{n} \gets $\PrimalWeightUpdate{$z^{n,0}, z^{n-1,0}, \omega^{n-1}$} \;
 }
 {\bf Output:} $z^{n,0}$.
 \caption{PDLP (after preconditioning and presolve)}
 \label{alg:practical-algorithm}
\end{algorithm}

\subsection{Step size choice}\label{sec:step-size-choice-body}

\begin{algorithm}
\SetAlgoLined
\Fn{\AdaptiveStepOfPDHG($z^{n, t}$, $\omega^n$, $\hat \eta^{n,t}, k$)}{
$(x, y) \gets z^{n, t}$, $\eta \gets \hat \eta^{n,t}$ \;
 \For{$i=1,\dots,\infty$}{
   $x' \gets \proj_{X}(x - \frac{\eta}{\omega^{n}} (c - K^\top y))$ \;
   $y' \gets \proj_{Y}(y + \eta\omega^{n} (q - K (2 x' - x)))$ \;
   $\bar \eta \gets \frac{\| (x'-x, y'-y) \|_{\omega^{n}}^2 }{ 2(y'-y)^\T K(x'-x)}$ \;
   $\eta' \gets \min\left(
      (1 - (k + 1)^{-0.3}) \bar \eta,
      ( 1 + (k + 1 )^{-0.6}) \eta\right)$ \;
   \If{$\eta \le \bar \eta$}{
     \Return{$(x', y')$, $\eta$, $\eta'$} 
   }
   $\eta \gets \eta'$ \;
 }
}
 \caption{One step of PDHG using our step size heuristic}
 \label{alg:step-size}
\end{algorithm}

The convergence analysis \cite[Equation (15)]{chambolle2016ergodic}
of PDHG (equation \eqref{eq:standard-pdhg})
 relies on a small constant step size
\begin{equation}
      \eta \le \frac{\| z^{k+1}-z^{k} \|_\omega^2}{2(y^{k+1}-y^k)^\T K(x^{k+1}-x^{k})} \label{eq:acceptableStep}
\end{equation}
where $z^{k} = (x^{k}, y^{k})$.
Classically one would ensure \eqref{eq:acceptableStep} by picking $\eta = \frac{1}{\|K\|_2}$. This is overly pessimistic and requires estimation of $\|K\|_2$. Instead our \AdaptiveStepOfPDHG adjusts $\eta$ dynamically to ensure that \eqref{eq:acceptableStep} is satisfied. If \eqref{eq:acceptableStep} isn't satisfied, we abort the step; i.e., we  reduce $\eta$, and try again. If \eqref{eq:acceptableStep} is satisfied we accept the step. This is described in Algorithm~\ref{alg:step-size}.
Note that in Algorithm~\ref{alg:step-size} $\bar \eta \ge \frac{1}{\|K\|_2}$ holds always, and from this one can show the resulting step size $\eta \ge \frac{1 - o(1)}{\|K\|_2}$ holds as $k \rightarrow \infty$.

Our step size routine compares favorably in practice with the line search by Malitsky and Pock \cite{malitsky2018linesearch} (See Appendix~\ref{sec:step-size-choice}).

\subsection{Adaptive restarts}\label{sec:adaptive-restarts}
In PDLP, we adaptively restart the PDHG algorithm in each outer iteration.
The key to our restarts at the $n$-th outer iteration is the normalized duality gap at $z$ which for any radius $r \in (0,\infty)$ is defined by
$$
\rho_{r}^n(z) := \frac{1}{r} \maximize_{(\hat{x},\hat{y}) \in \{  \hat{z} \in Z : \| \hat{z} - z \|_{\omega^{n}} \le r \}}\{ \Lag(x, \hat{y}) - \Lag(\hat{x}, y)\} ,
$$
introduced by \cite{applegate2021restarts}. Unlike the standard duality gap
$$
\maximize_{(\hat{x},\hat{y}) \in Z} \{ \Lag(x, \hat{y}) - \Lag(\hat{x}, y)\},
$$ 
the normalized duality gap is always a finite quantity. Furthermore, for any value of $r$ and $\omega^n$, the normalized duality gap $\rho_r^n(z)$ is $0$ if and only if the solution $z$ is an optimal solution to \eqref{eq:primal-dual} \cite{applegate2021restarts}; thus, it provides a valid metric for measuring progress towards the optimal solution. The normalized duality gap is computable in linear time \cite{applegate2021restarts}. For brevity, define $\mu_n(z, z_{\text{ref}})$ as the normalized duality gap at $z$ with radius $\| z -  z_{\text{ref}} \|_{\omega^{n}}$, i.e.,
$$
\mu_n(z, z_{\text{ref}}) := \rho_{\| z -  z_{\text{ref}} \|_{\omega^{n}}}^n(z) ,
$$
where $z_{\text{ref}}$ is a user-chosen reference point.

\paragraph{Choosing the restart candidate.}
To choose the restart candidate $\zc^{n,t+1}$ we call %
$$
\text{\GetRestartCandidate{$z^{n,t+1}, \bz^{n,t+1}, z^{n,0}$}} := \begin{cases}
z^{n,t+1} & \mu_{n}(z^{n,t+1}, z^{n,0}) < \mu_n(\bz^{n,t+1}, z^{n,0}) \\
\bz^{n,t+1} & \text{otherwise}  \ .
\end{cases}
$$
This choice is justified in Remark~5 of \cite{applegate2021restarts}.

\paragraph{Restart criteria.} We define three parameters: $\beta_{\text{sufficient}} \in (0,1)$, $\beta_{\text{necessary}} \in (0, \beta_{\text{sufficient}})$ and $\beta_{\text{artificial}} \in (0,1)$. In PDLP we use $\beta_{\text{sufficient}} = 0.9$, $\beta_{\text{necessary}} = 0.1$, and $\beta_{\text{artificial}} = 0.5$. The algorithm restarts if one of three conditions holds: 

\begin{itemize}[leftmargin=.7cm]
    \item[(i)] (\textbf{Sufficient decay in normalized duality gap})
$
\mu_n(\zc^{n,t+1}, z^{n,0}) \le \beta_{\text{sufficient}} \mu_n(z^{n,0}, z^{n-1,0}) \ ,
$
    \item[(ii)] (\textbf{Necessary decay + no local progress in normalized duality gap}) 
    \begin{equation*}
        \mu_n(\zc^{n,t+1}, z^{n,0}) \le \beta_{\text{necessary}} \mu_n(z^{n,0}, z^{n-1,0}) \quad \text{and} \quad \mu_n(\zc^{n,t+1}, z^{n,0}) > \mu_n(z^{n,t}_c, z^{n,0}) \ ,
    \end{equation*}
        \item[(iii)]  (\textbf{Long inner loop})
    $t \ge \beta_{\text{artificial}} k\ .$
\end{itemize}
The motivation for (i) is presented in \cite{applegate2021restarts}; it guarantees the linear convergence of restarted PDHG on LP problems. 
The second condition in (ii) is inspired by adaptive restart schemes for accelerated gradient descent where restarts are triggered if the function value increases \cite{o2015adaptive}.
The first inequality in (ii) provides a safeguard for the second one, preventing the algorithm restarting every inner iteration or never restarting.
The motivation for (iii) relates to the primal weights (Section~\ref{sec:primal-weight-update}). In particular, primal weight updates only occur after a restart, and condition (iii) ensures that the primal weight will be updated infinitely often.
This prevents a bad choice of primal weight in earlier iterations causing progress to stall for a long time.

\subsection{Primal weight updates}\label{sec:primal-weight-update}

The primal weight is initialized using
\begin{flalign*}
\InitializePrimalWeight{c, q} := \begin{cases} \frac{\| c \|_2}{\| q \|_2} & \| c \|_2, \| q \|_2 > \epsZero \\
1 & \text{otherwise}
\end{cases}
\end{flalign*}
where $\epsZero$ is a small nonzero tolerance.
This primal weight update scheme guarantees scale invariance.
In particular, in Appendix~\ref{app:scale-invariance} we consider PDHG with $\epsZero = 0$, $\eta = 0.9/\| K \|_2$ and $\omega = \InitializePrimalWeight{c, q}$. 
In this simplified setting, we prove that if we multiply the objective, constraints, or the right hand side and variable bounds by a scalar then the iterate behaviour remain identical (up to a scaling factor). 

\begin{algorithm}
\caption{Primal weight update}\label{alg:primal-weight-update}
\SetAlgoLined
\Fn{\PrimalWeightUpdate{$z^{n,0}, z^{n-1,0}, \omega^{n-1}$}}{
$\Delta_x^n = \| x^{n,0} - x^{n-1,0} \|_2, \quad \Delta_y^n = \| y^{n,0} - y^{n-1,0} \|_2$ \;
\uIf{$\Delta_x^n > \epsZero$ and $\Delta_y^n > \epsZero$}{
\Return{ $\exp\left(\theta \log\left( \frac{\Delta_y^n}{\Delta_x^n} \right) + (1 - \theta) \log(\omega^{n-1}) \right)$}
}
\Else{
\Return{$\omega^{n-1}$} \;
}
}
\end{algorithm}

Algorithm~\ref{alg:primal-weight-update} aims to choose the primal weight $\omega^{n}$ such that distance to optimality in the primal and dual is the same, i.e.,
$\| (x^{n,t} - \xStar, \mathbf{0}) \|_{\omega^n} \approx \| (\mathbf{0}, y^{n,t} - \yStar) \|_{\omega^n}$. By definition of $\| \cdot \|_{\omega}$,
$$
\| (x^{n,t} - \xStar, \mathbf{0}) \|_{\omega^n} = \omega^n \| x^{n,t} - \xStar \|_2, \quad \| (\mathbf{0}, y^{n,t} - \yStar) \|_{\omega^n} = \frac{1}{\omega^n} \| y^{n,t} - \yStar \|_2.
$$
Setting these two terms equal yields
$\omega^n = \frac{\| y^{n,t} - \yStar \|_2}{\| x^{n,t} - \xStar \|_2}$.
Of course, the quantity $\frac{\| y^{n,t} - \yStar \|_2}{\| x^{n,t} - \xStar \|_2}$ is unknown beforehand, but we attempt to estimate it using $\Delta^n_y / \Delta^n_x$. However, the quantity $\Delta^n_y / \Delta^n_x$ can change wildly from one restart to another, causing $\omega^n$ to oscillate. To dampen variations in $\omega^n$, we first move to a log-scale where the primal weight is symmetric, i.e., $\log(1 / \omega^n) = -\log(\omega^n)$, and perform a exponential smoothing with parameter $\theta \in [0,1]$. In PDLP, we use $\theta = 0.5$. 

There are several important differences between our primal weight heuristic and literature \cite{goldstein2013adaptive,goldstein2015adaptive}. For example, \cite{goldstein2013adaptive,goldstein2015adaptive} make relatively small changes to the primal weights at each iteration, attempting to balance the primal and dual residual. These changes have to be diminishingly small because, in our experience, PDHG may be unstable if they are too big.  In contrast, in our method the primal weight is only updated during restarts, which in practice allows for much larger changes without instability issues. Moreover, our scheme tries to balance the weighted distance traveled in the primal and dual rather than the residuals \cite{goldstein2013adaptive,goldstein2015adaptive}.

\subsection{Presolve}\label{sec:presolve}

Presolving refers to transformation steps that simplify the input problem before starting the optimization solver. These steps span from relatively easy transformations such as detecting inconsistent bounds, removing empty rows and columns of $K$, and removing variables whose lower and upper bounds are equal, to more complex operations such as detecting duplicate rows in $K$ and tightening bounds. Presolve is a standard component of traditional LP solvers~\cite{maros2002computational}. We are not aware of presolve being combined with PDHG for LP. However, \cite{lin2021admm, li2020asymptotically} combine presolve with other FOMs. 

As an experiment to measure the impact of presolve, we used PaPILO \cite{gamrath2020scip}, an open-source presolving library. For technical reasons, it was easier to use PaPILO as a standalone executable than as a library. We simulate its effect by simply solving the preprocessed instances. Convergence criteria are evaluated with respect to the presolved instance, not the original problem.

\subsection{Diagonal Preconditioning}\label{sec:diag}

Preconditioning is a popular heuristic in  optimization for improving the convergence of FOMs.
To avoid factorizations, we only consider diagonal preconditioners. Our goal is to rescale the constraint matrix $K=(G, A)$ to $\tK=(\tG, \tA)=D_1 K D_2$ with positive diagonal matrices $D_1$ and $D_2$, so that the resulting matrix $\tK$ is ``well balanced''. Such preconditioning creates a new LP instance that replaces $A, G, c, b, h, u$, and $l$ in \eqref{eq:lp} with
$\tG, \tA$, $\tx = D_2^{-1} x$, $\tc=D_2 c$, $(\tb,\thh)= D_1 (b,h)$, $\tu = D_2^{-1} u$ and $\tl = D_2^{-1} l$. Common choices for $D_1$ and $D_2$ include:
\begin{itemize}[leftmargin=*]
    \item \textbf{No scaling:} Solve the original LP instance \eqref{eq:lp} without additional scaling, namely $D_1=D_2=I$.
    \item \textbf{Pock-Chambolle \cite{pock2011diagonal}:} Pock and Chambolle proposed a family of diagonal preconditioners\footnote{Diagonal preconditioning is equivalent to changing to a weighted $\ell_2$ norm in the proximal step of PDHG (weight defined by $D_2$ and $D_1$ for the primal and dual respectively). Pock and Chambolle use this weighted norm perspective.} for PDHG parameterized by $\alpha$, where the diagonal matrices are defined by $(D_1)_{jj}=\sqrt{\|K_{j,\cdot}\|_{2-\alpha}}$ for $j=1,...,m_1+m_2$ and $(D_2)_{ii}=\sqrt{\|K_{\cdot,i}\|_{\alpha}}$ for $i=1,...,n$. We use $\alpha=1$ in PDLP (we also tested $\alpha=0$ and $\alpha=2$). This is the baseline diagonal preconditioner in the PDHG literature.
    \item \textbf{Ruiz~\cite{ruiz2001scaling}:} Ruiz scaling is a popular algorithm in numerical linear algebra to equilibrate matrices. In an iteration of Ruiz scaling, the diagonal matrices are defined as $(D_1)_{jj}=\sqrt{\|K_{j,\cdot}\|_{\infty}}$ for $j=1,...,m_1+m_2$ and $(D_2)_{ii}=\sqrt{\|K_{\cdot,i}\|_{\infty}}$ for $i=1,...,n$. Ruiz \cite{ruiz2001scaling} shows that if this rescaling is applied iteratively, the infinity norm of each row and each column converge to $1$.
\end{itemize}

For the default PDLP settings, we apply a combination of Ruiz rescaling~\cite{ruiz2001scaling} and the preconditioning technique proposed by Pock and Chambolle \cite{pock2011diagonal}. In particular, we apply $10$ iterations of Ruiz scaling and then apply the Pock-Chambolle scaling. To illustrate the effectiveness of our proposed scaling technique, we compare it against these three common techniques in Appendix \ref{sec:app-diagonal}.

\subsection{Theoretical guarantees for the above enhancements}
\label{sec:guarantees}
While PDLP's enhancements are motivated by theory, some of them may not preserve theoretical guarantees as discussed below:
\begin{itemize}
    \item We do not have a proof of convergence for the adaptive step size rule (Section \ref{sec:step-size-choice-body}).
    \item One can show our restart criteria (Section \ref{sec:adaptive-restarts}) preserve convergence guarantees by modifying the proof of \cite{applegate2021restarts} to a more general setting.
    \item Primal weight updates (Section \ref{sec:primal-weight-update}) do not readily preserve convergence guarantees, but we conjecture that a proof of convergence is possible if they are updated infrequently.
    \item Presolve (Section \ref{sec:presolve}) and diagonal preconditioning (Section \ref{sec:diag}) preserve theoretical guarantees because they can be viewed as applying PDHG to an LP instance with different data. 
\end{itemize}

\section{Numerical experiments}
\label{sec:experiments}
Our numerical experiments study the effectiveness of PDLP primarily with respect to traditional LP applications and benchmark sets.
Section~\ref{sec:setup} describes the setup for the experiments. Section~\ref{sec:improvement-measured-impact} demonstrates PDLP's improvements over baseline PDHG. Section~\ref{sec:baselines} compares PDLP with other FOMs. Section~\ref{sec:medium-gurobi} highlights benchmark instances where PDLP outperforms a commercial LP solver. Finally, Section~\ref{sec:pagerank} illustrates the ability of PDLP to scale to a large application where barrier and simplex-based solvers run out of memory. The supplemental materials contain extensive ablation studies and additional instructions for reproducing the experiments.

\subsection{Experimental setup}\label{sec:setup}

\paragraph{Optimality termination criteria.} PDLP terminates with an approximately optimal solution when the primal-dual iterates $x \in X$, $y \in Y$, $\lambda \in \Lambda$, satisfy:
\begin{subequations}\label{eq:opt-termination}
\begin{flalign}
\label{eq:dual-gap-feas}
\abs{q^\T y + l^\T \lambda^{+} - u^\T \lambda^{-} - c^\T x  } &\le \epsilon  (1 + \abs{ q^\T y + l^\T \lambda^{+} - u^\T \lambda^{-} } + \abs{c^\T x}) \\
\label{eq:primal-feas}
\left\| \begin{pmatrix} A x - b \\
(h - G x)^{+} \end{pmatrix} \right\|_2 &\le \epsilon  (1 + \| q \|_2) \\
\label{eq:dual-feas}
\| c - K^\T y - \lambda \|_2 &\le \epsilon (1 + \| c \|_2)
\end{flalign}
\end{subequations}
where $\epsilon \in (0,\infty)$ is the termination tolerance. Note that if \eqref{eq:opt-termination} is satisfied with $\epsilon = 0$, then by LP duality we have found an optimal solution \cite{goldmantucker}. Indeed, \eqref{eq:dual-gap-feas} is the duality gap, \eqref{eq:primal-feas}
is primal feasibility, and \eqref{eq:dual-feas} is dual feasibility. We use these criteria to be consistent with those of SCS~\cite{scs}. The PDHG algorithm does not explicitly include a reduced costs variable $\lambda$. Therefore,
to evaluate the optimality termination criteria we compute $\lambda = \proj_{\Lambda}(c - K^\T y)$. All instances considered have an optimal primal-dual solution. We use $\epsilon = 10^{-8}$ as a benchmark for high-quality solutions and $\epsilon = 10^{-4}$ for moderately accurate solutions.

\paragraph{Benchmark datasets.}We use three datasets to compare algorithmic performance.
One is the \texttt{LP benchmark} dataset of 56 problems, formed by merging the instances from ``Benchmark of Simplex LP Solvers'', ``Benchmark of Barrier LP solvers'', and ``Large Network-LP Benchmark'' from \cite{mittelmannbenchmark}.
We also created a larger benchmark of 383 instances curated from LP relaxations of mixed-integer programming problems from the MIPLIB2017 collection~\cite{gleixner2021miplib} (see Appendix~\ref{sec:currated-miplib-dataset}) that we label \texttt{MIP Relaxations}. \texttt{MIP Relaxations} was used extensively during algorithmic development, e.g., for hyperparameter choices; we held out \texttt{LP benchmark} as a test set.
Finally, we also performed some experiments on the \texttt{Netlib} LP benchmark~\cite{netlib}, an historically important benchmark that is no longer state of the art for large-scale LP.

\paragraph{Software.} PDLP is implemented in an open-source Julia~\cite{bezanson2017julia} module available at \url{https://github.com/google-research/FirstOrderLp.jl}. The module also contains a baseline implementation of the extragradient method with many of the same enhancements as PDLP (labeled `Enh. Extragradient'). We compare with two external packages: SCS~\cite{scs} version 2.1.3, an open-source generic cone solver based on ADMM, and Gurobi version 9.0.1, a state-of-the-art commercial LP solver. SCS supports two modes for solving the linear system that arises at each iteration, a direct method based on a cached LDL factorization (which is the default `SCS') and an indirect method based on the conjugate gradient method (which we label `SCS (matrix-free)'). All solvers are run single-threaded. SCS and Gurobi are provided the same presolved instances as PDLP.

\paragraph{Computing environment.} We used two computing environments for our experiments: 1) \texttt{e2-highmem-2} virtual machines (VMs) on Google Cloud Platform (GCP). Each VM provides two virtual CPUs and 16GB RAM. 2) A dedicated workstation with an Intel Xeon E5-2669 v3 processor and 128 GB RAM. This workstation has a license for Gurobi that permits at most one concurrent solve. Total compute time on GCP for all preliminary and final experiments was approximately $72,000$ virtual CPU hours.

\paragraph{Initialization.} All first-order methods use all-zero vectors as the initial starting points.

\paragraph{Metrics.} We use the term \textit{KKT passes} to refer to the number of matrix multiplications by both $K$ and $K^\T$.
Given that the most expensive operation in our algorithm is matrix-vector multiplication, this metric is less noisy than runtime for comparing performance between matrix-free solvers. 
SGM10 stands for shifted geometric mean with shift $10$, which is computed by adding $10$ to all data points, taking the geometric mean, and then subtracting $10$. Unsolved instances are assigned values corresponding to the limits specified in the next paragraph.

\paragraph{Time and KKT pass limits.} For Section~\ref{sec:improvement-measured-impact} we impose a limit on the KKT passes of $100,000$. For Section~\ref{sec:baselines} we impose a time limit of 1 hour.

\begin{figure}
\begin{subfigure}{.33\textwidth}
\centering\includegraphics[width=1.0\linewidth]{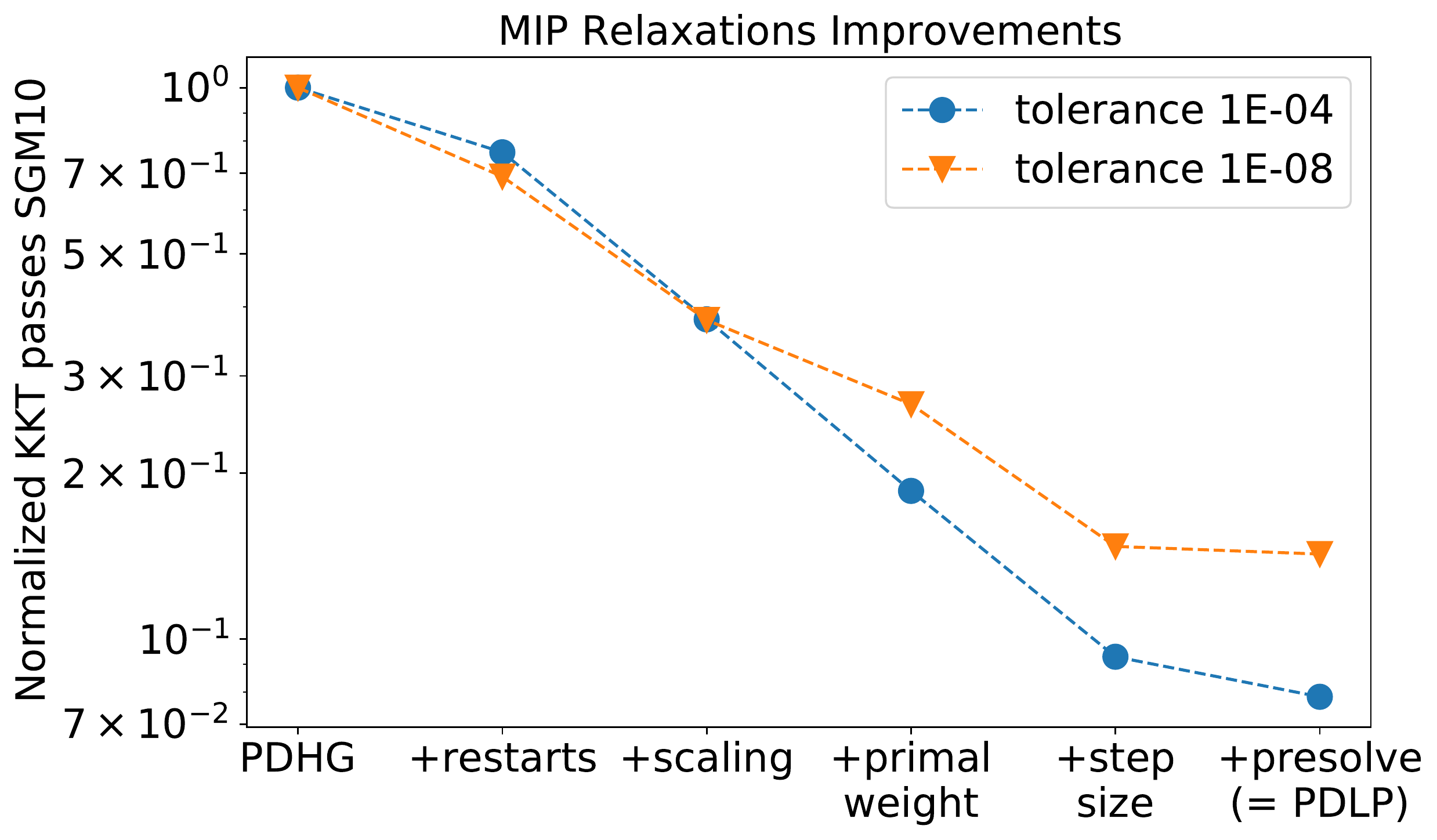}
\caption{\texttt{MIP Relaxations}}\label{fig:miplib-improvements}
\end{subfigure}
\begin{subfigure}{.33\textwidth}
\centering\includegraphics[width=1.0\linewidth]{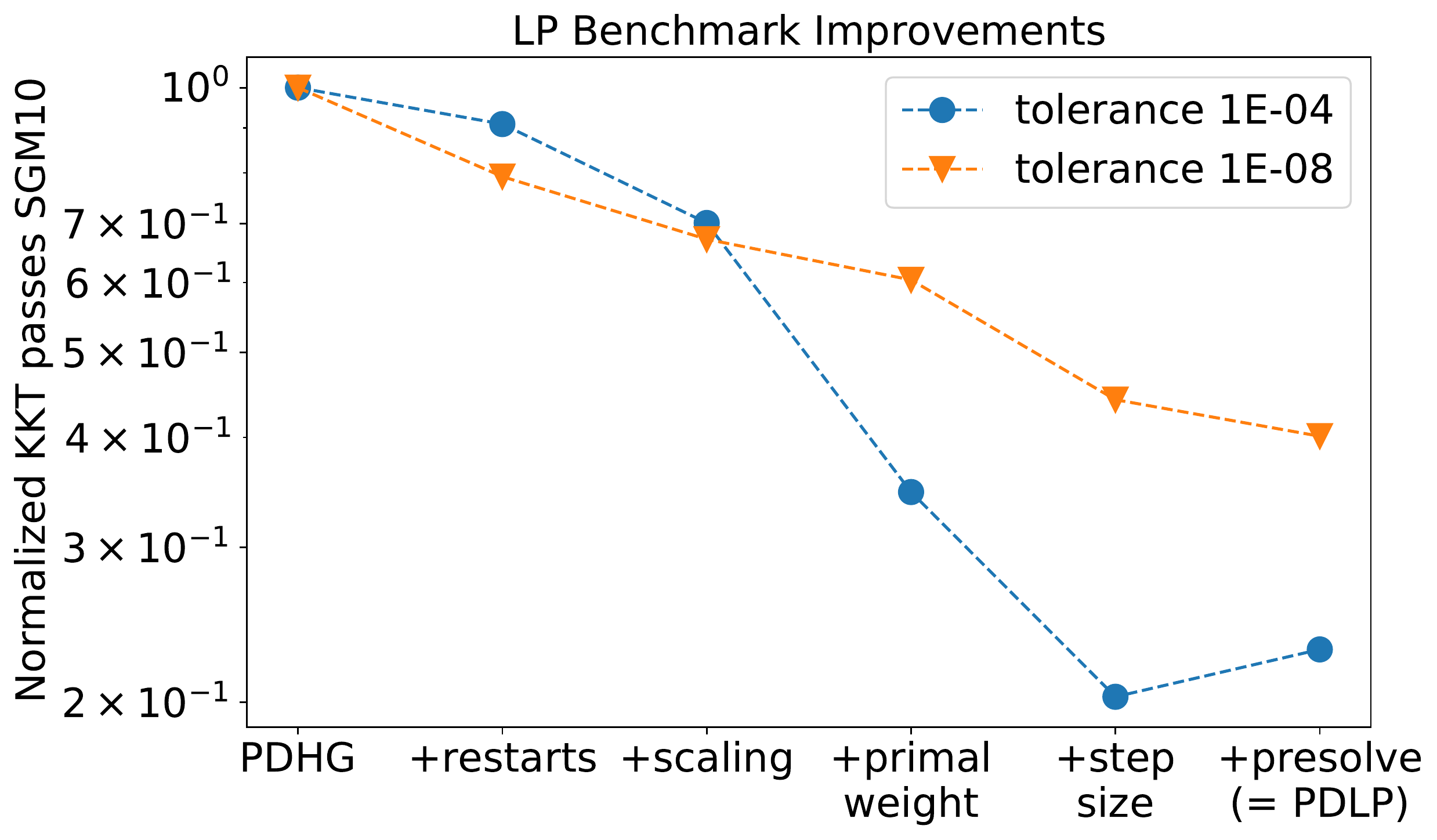}
\caption{\texttt{LP benchmark}}\label{fig:lpbenchmark-improvements}
\end{subfigure}
\begin{subfigure}{.33\textwidth}
\centering\includegraphics[width=1.0\linewidth]{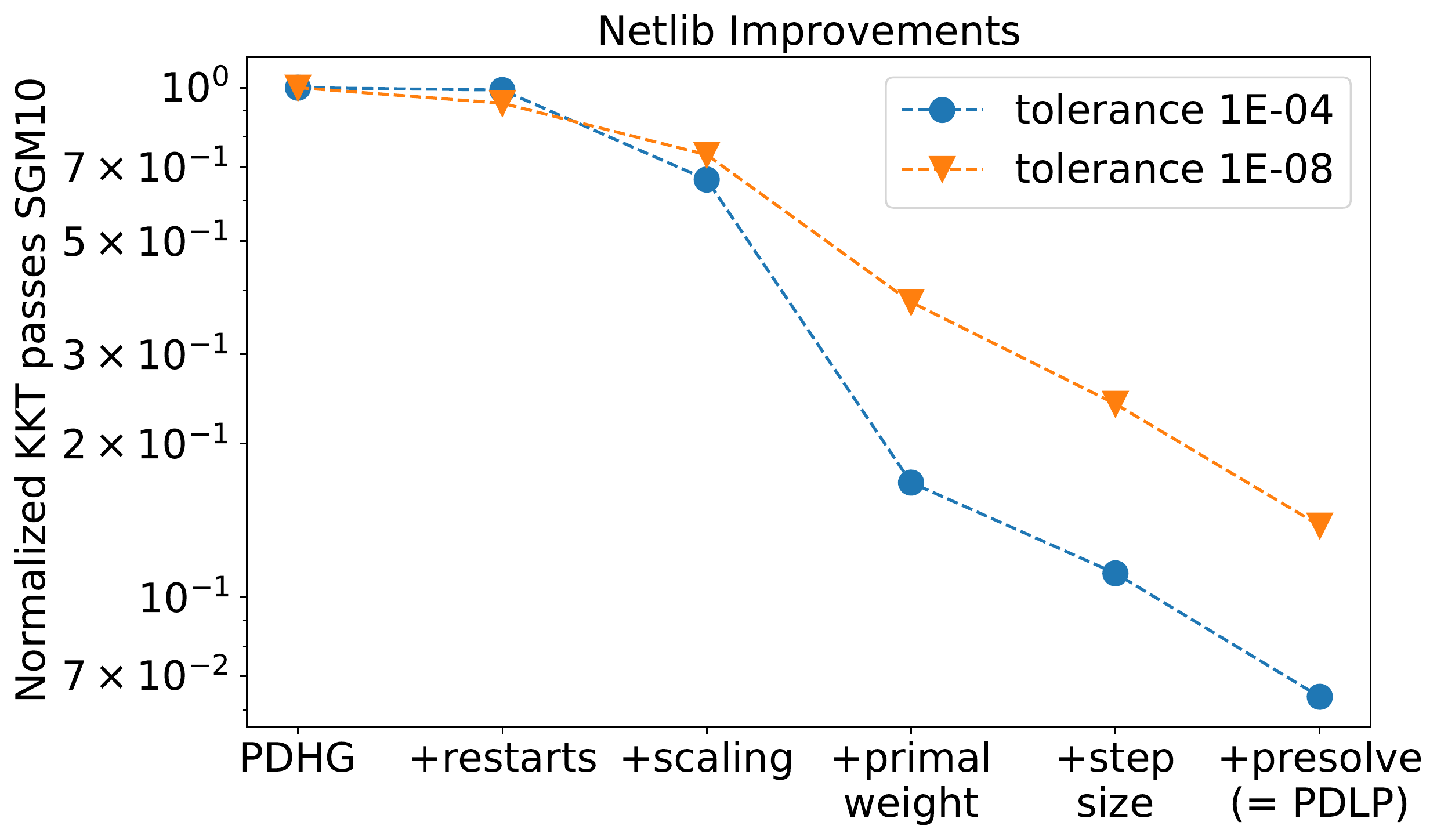}
\caption{\texttt{Netlib}}\label{fig:netlib-improvements}
\end{subfigure}
\caption{Summary of relative impact of PDLP's improvements}
\label{fig:pdlp-improvements}
\end{figure}

\subsection{Impact of PDLP's improvements}\label{sec:improvement-measured-impact}

The y-axes of Figure~\ref{fig:pdlp-improvements}
display the SGM10 of the KKT passes normalized by the value for baseline PDHG.
We can see, with the exception of presolve for \texttt{LP benchmark} at tolerance $10^{-4}$, each of our modifications described in Section~\ref{sec:practical-algorithmic-improvements} improves the performance of PDHG.

\subsection{Comparison with other first-order baselines}\label{sec:baselines}

We compared PDLP with several other first-order baselines: SCS~\cite{scs}, in both direct (default) mode and matrix-free mode, and our enhanced implementation of the extragradient method \cite{korpelevich1976extragradient,nemirovski2004prox}. For SCS in matrix-free mode, we include the KKT passes from the conjugate gradient solves; for SCS in direct mode there is no reasonable measure of KKT passes for the factorization and direct solve, so we only measure running time.  The comparisons are summarized in Figure~\ref{fig:pdlp-vs-baseline}.

\begin{figure}
\begin{subfigure}{.24\textwidth}
  \centering
  \includegraphics[width=.95\linewidth]{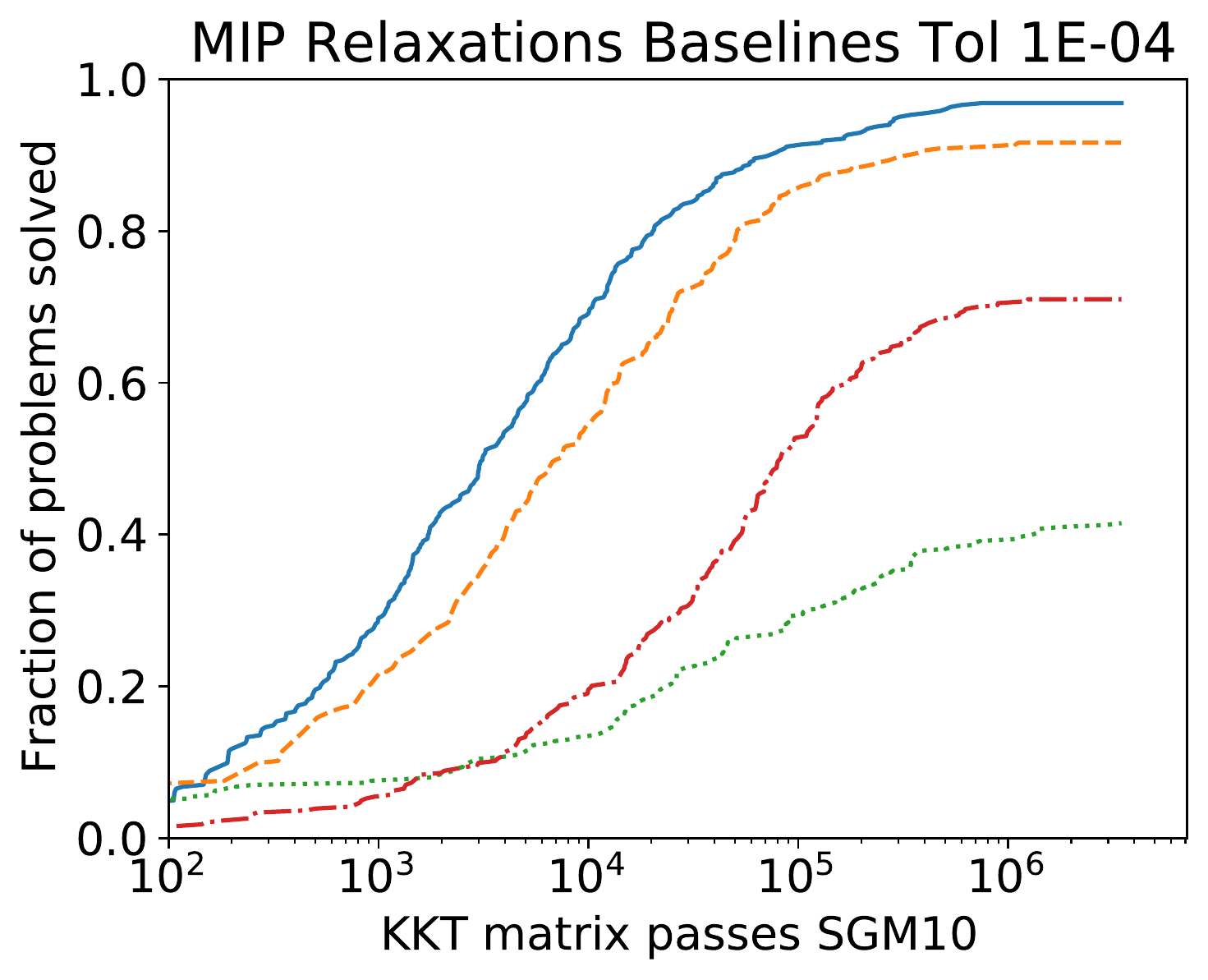}
\end{subfigure}%
\begin{subfigure}{.24\textwidth}
  \centering
  \includegraphics[width=.95\linewidth]{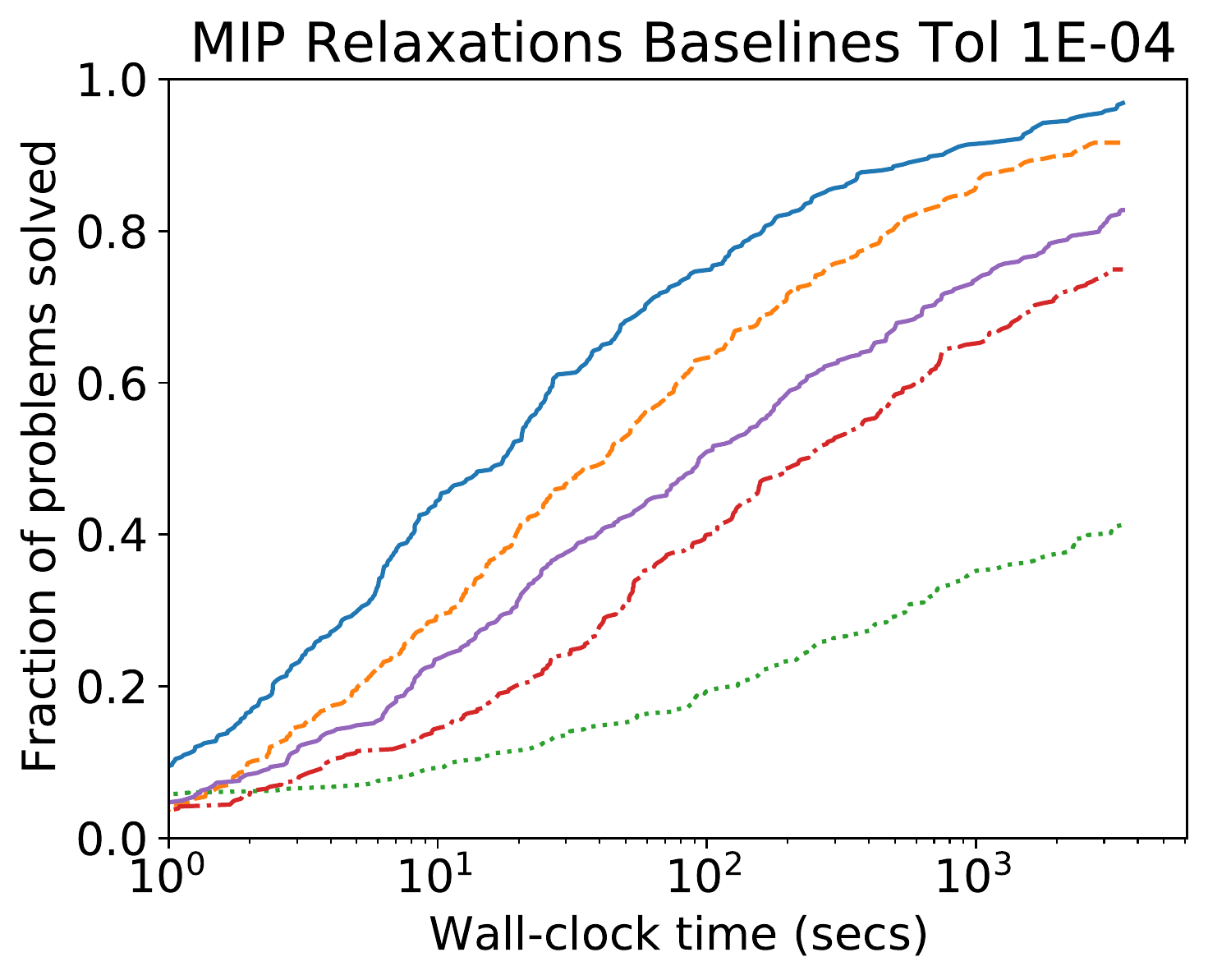}
\end{subfigure}
\begin{subfigure}{.24\textwidth}
  \centering
  \includegraphics[width=.95\linewidth]{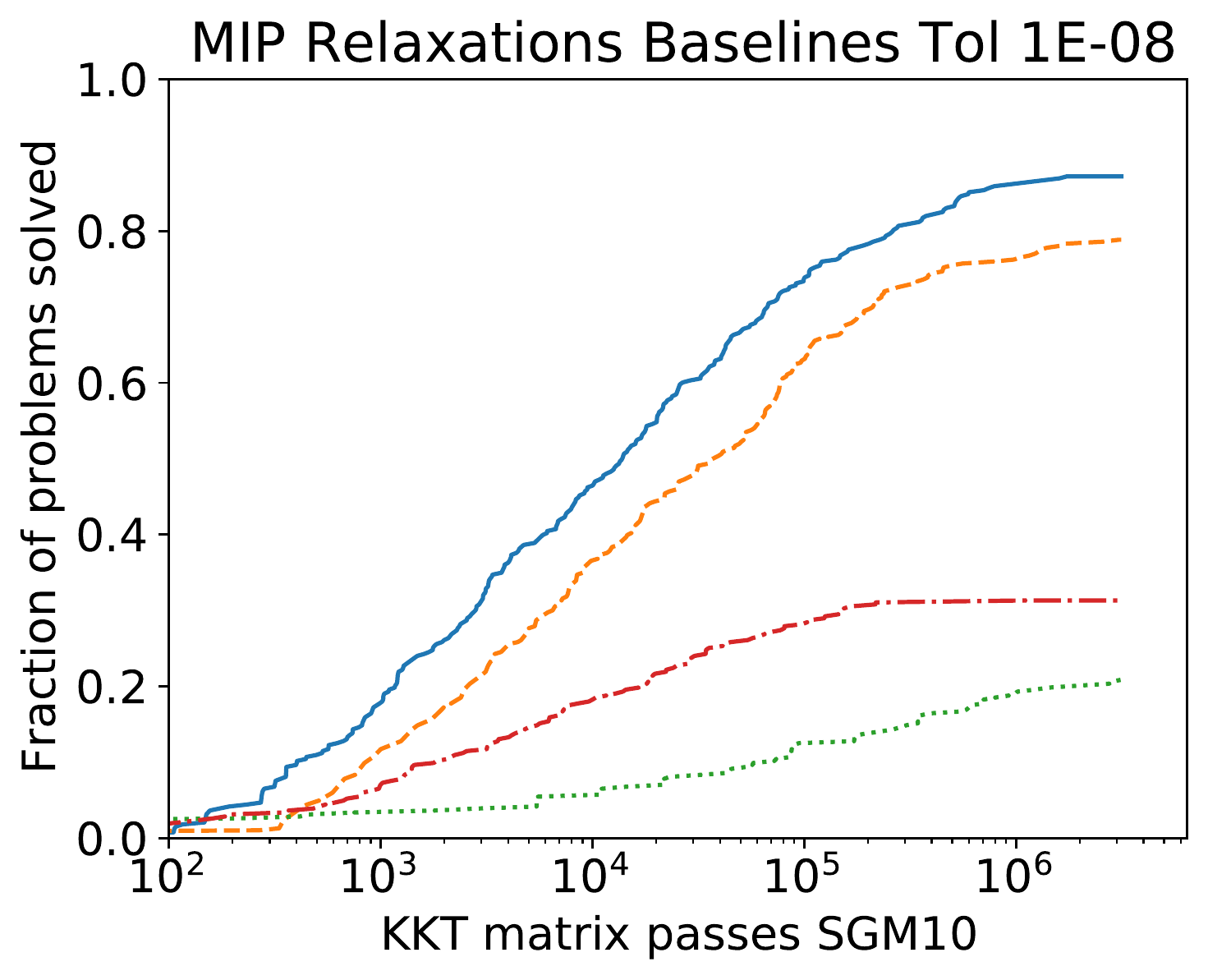}
\end{subfigure}
\begin{subfigure}{.24\textwidth}
  \centering
  \includegraphics[width=.95\linewidth]{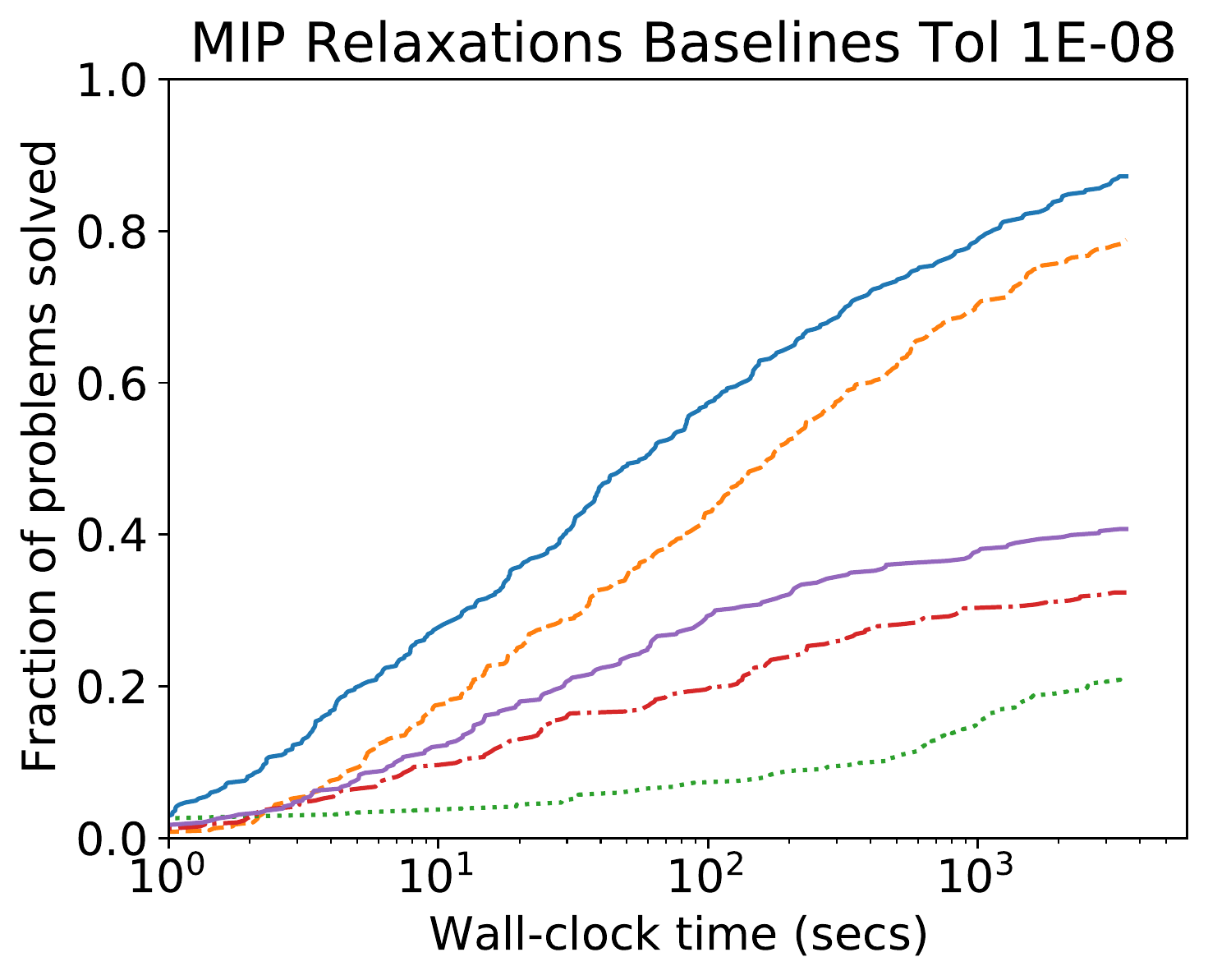}
\end{subfigure}
\\
\begin{subfigure}{.24\textwidth}
  \centering
  \includegraphics[width=.95\linewidth]{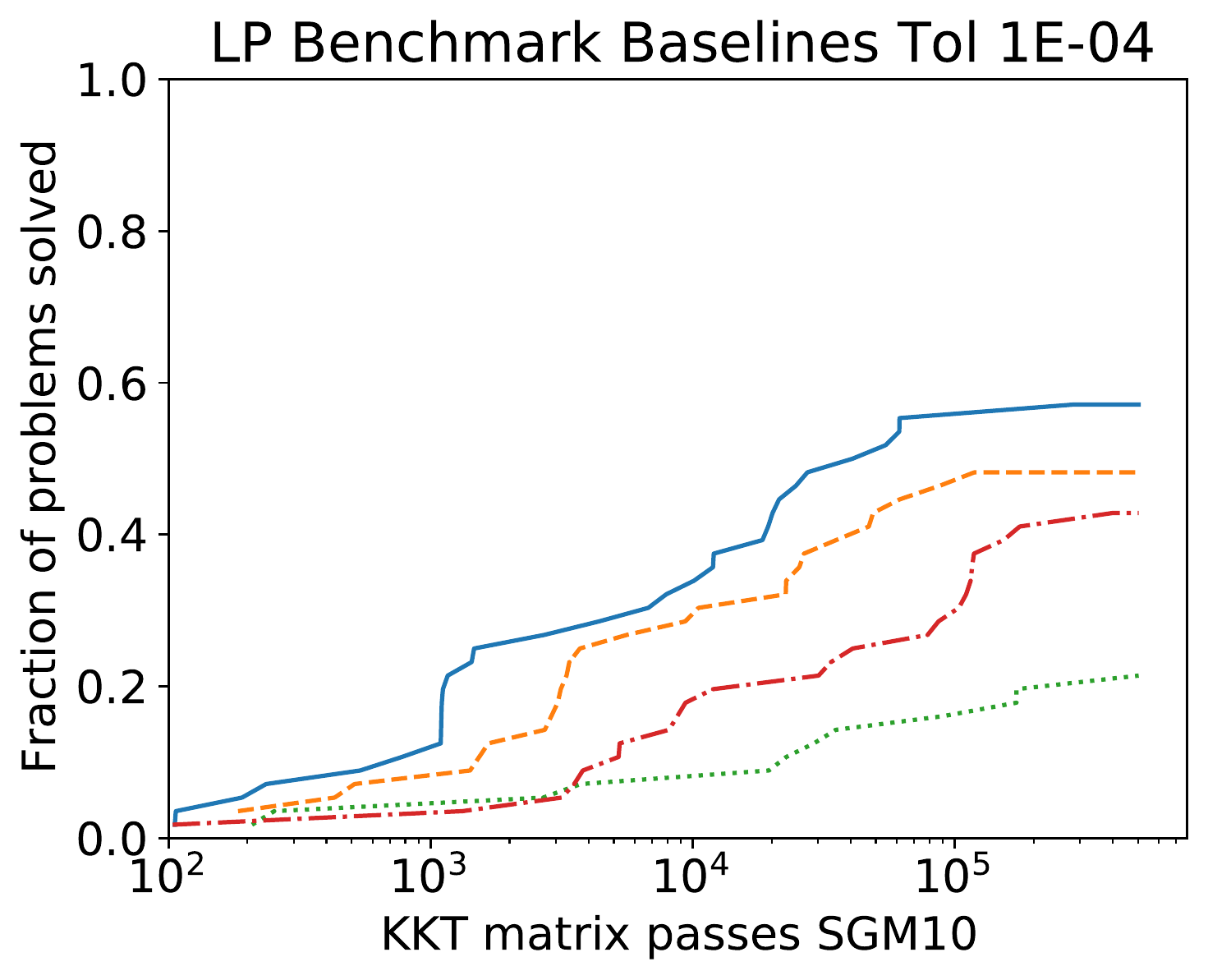}
\end{subfigure}%
\begin{subfigure}{.24\textwidth}
  \centering
  \includegraphics[width=.95\linewidth]{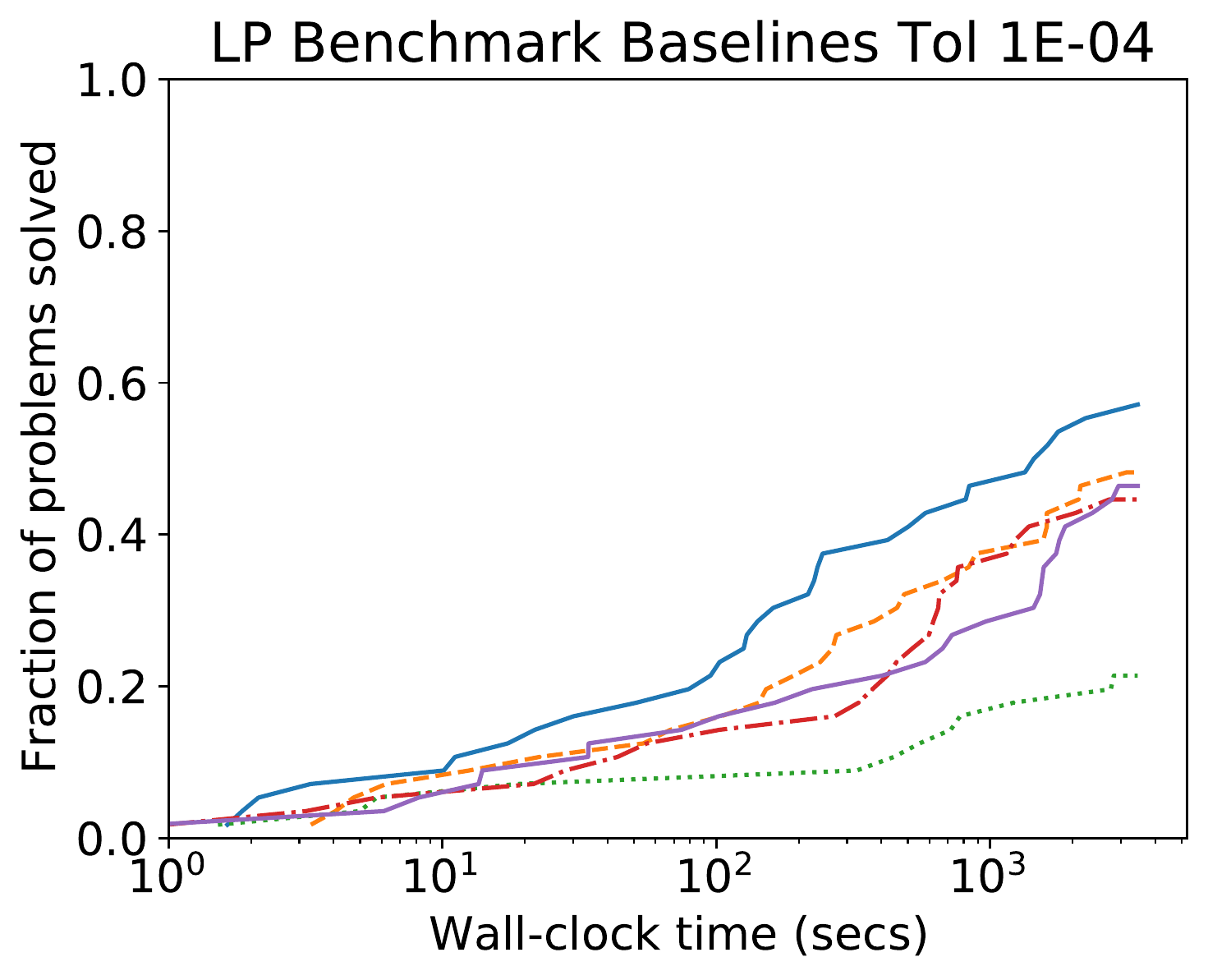}
\end{subfigure}
\begin{subfigure}{.24\textwidth}
  \centering
  \includegraphics[width=.95\linewidth]{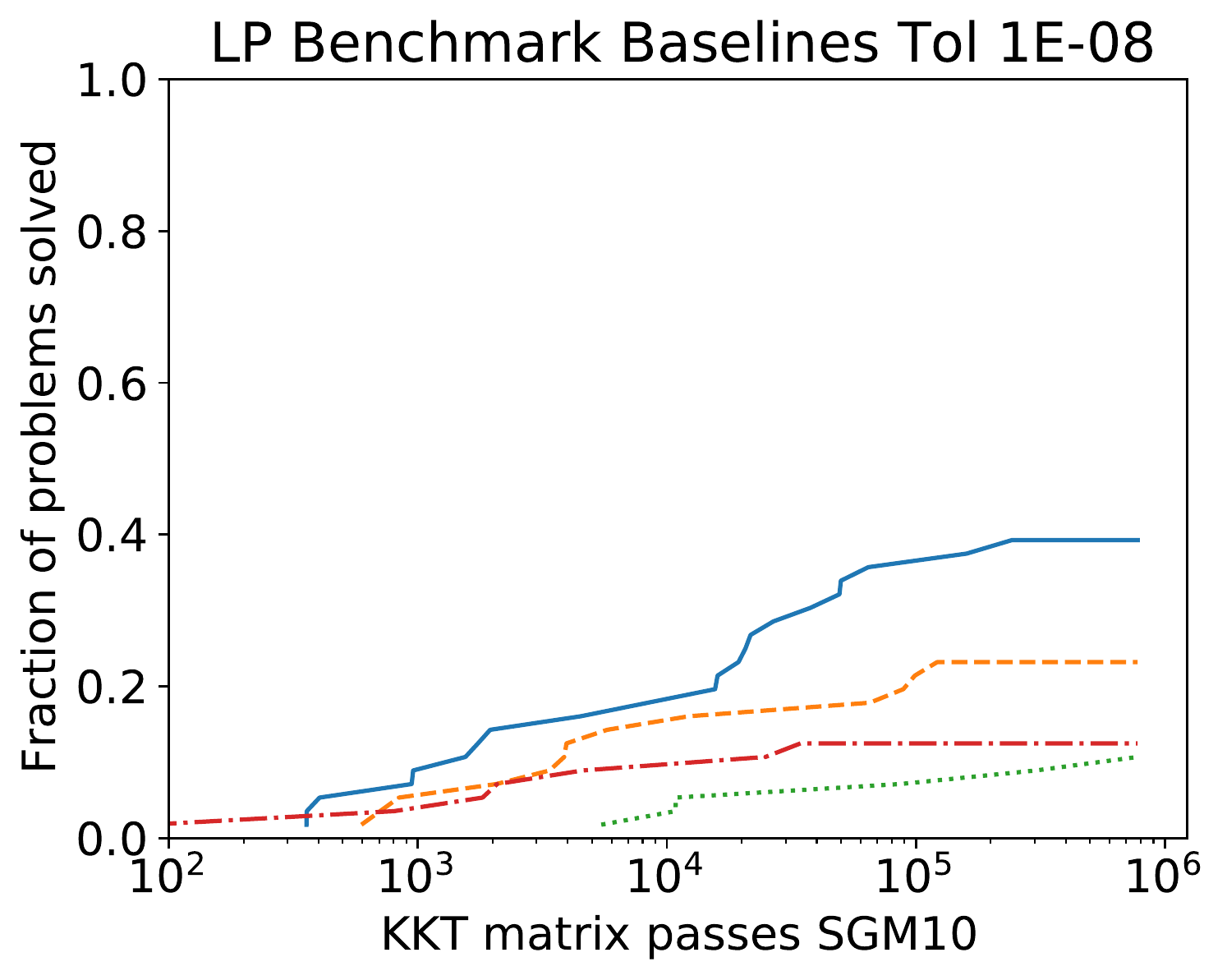}
\end{subfigure}
\begin{subfigure}{.24\textwidth}
  \centering
  \includegraphics[width=.95\linewidth]{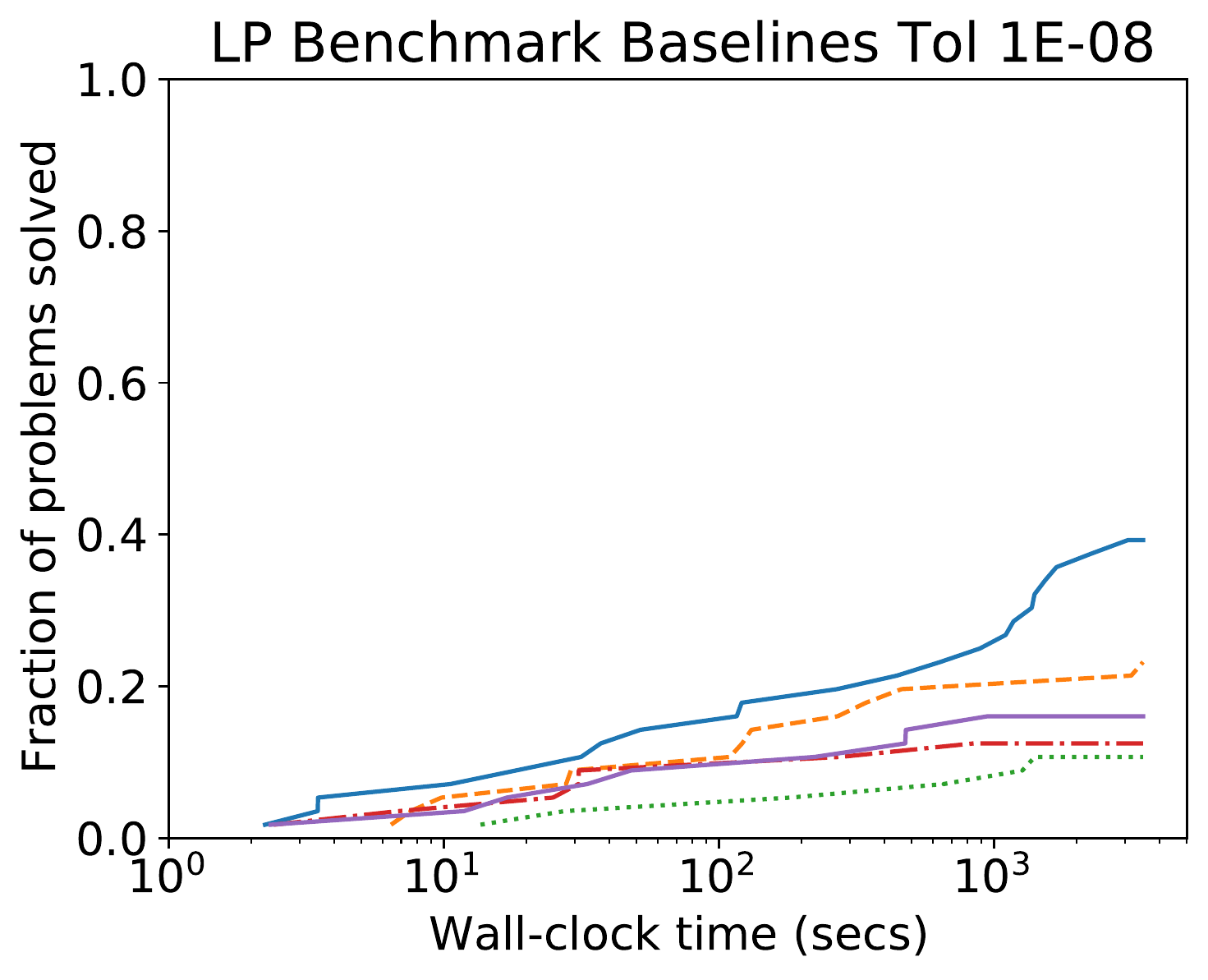}
\end{subfigure}
\\
\begin{subfigure}{.24\textwidth}
  \centering
  \includegraphics[width=.95\linewidth]{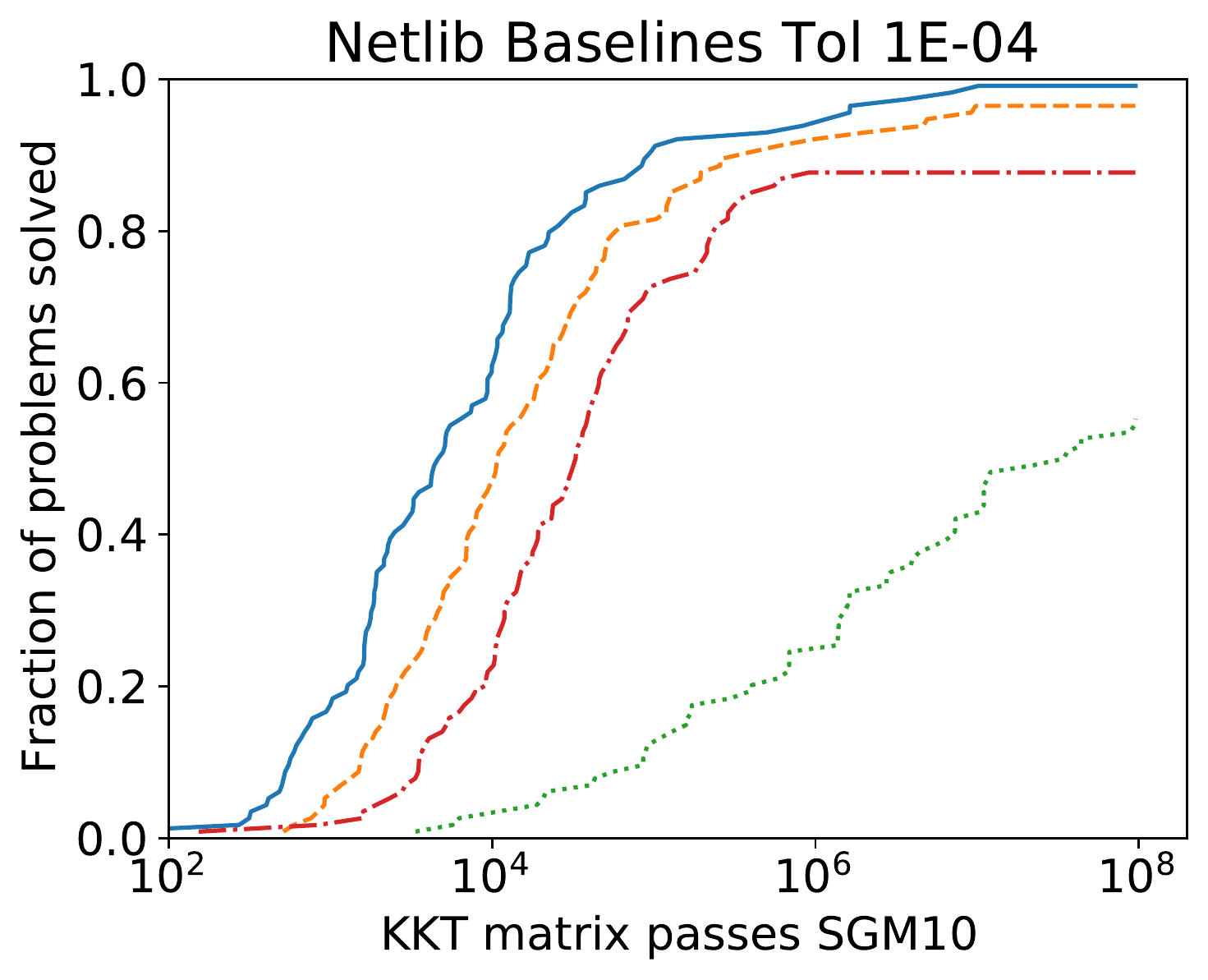}
\end{subfigure}%
\begin{subfigure}{.24\textwidth}
  \centering
  \includegraphics[width=.95\linewidth]{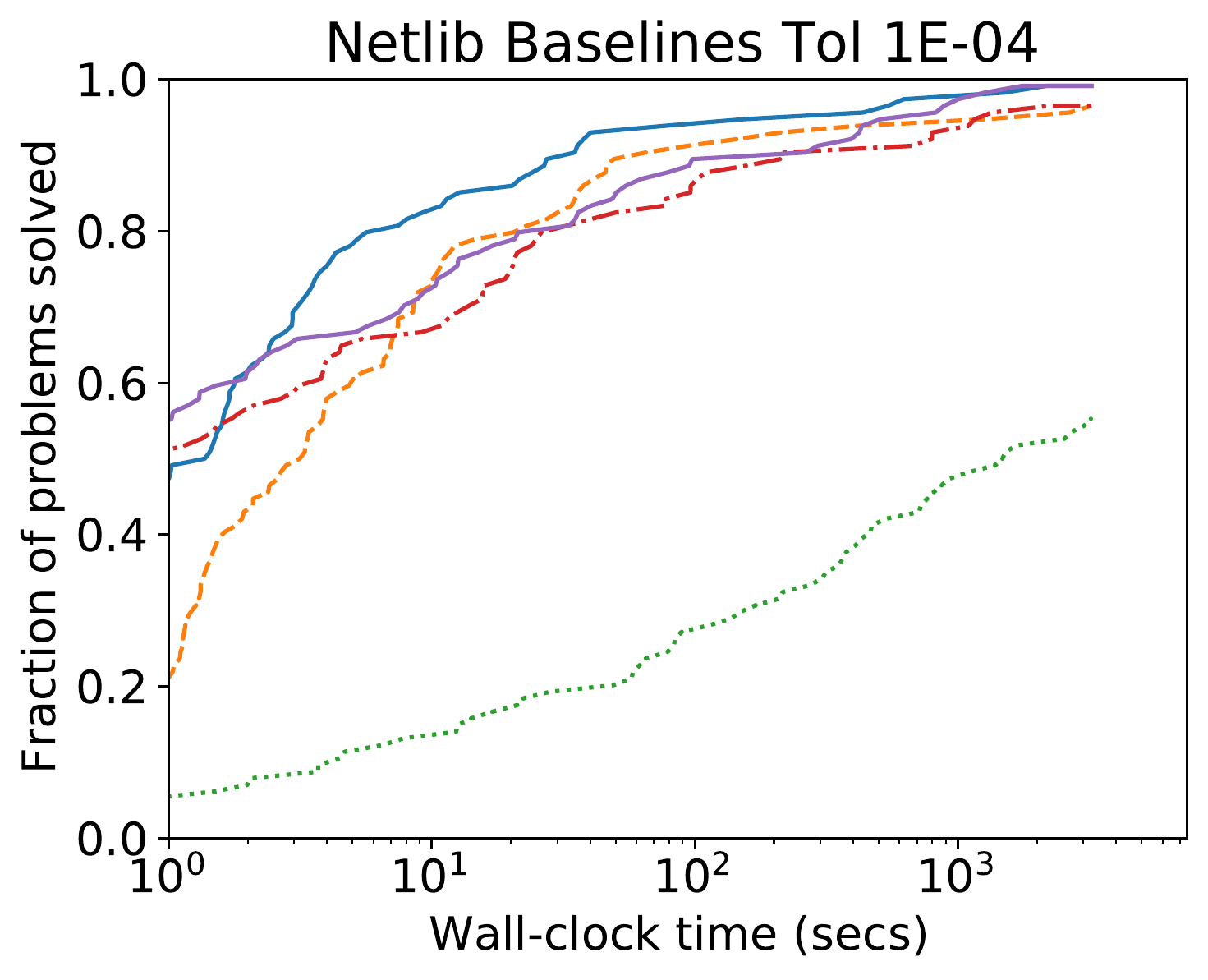}
\end{subfigure}
\begin{subfigure}{.24\textwidth}
  \centering
  \includegraphics[width=.95\linewidth]{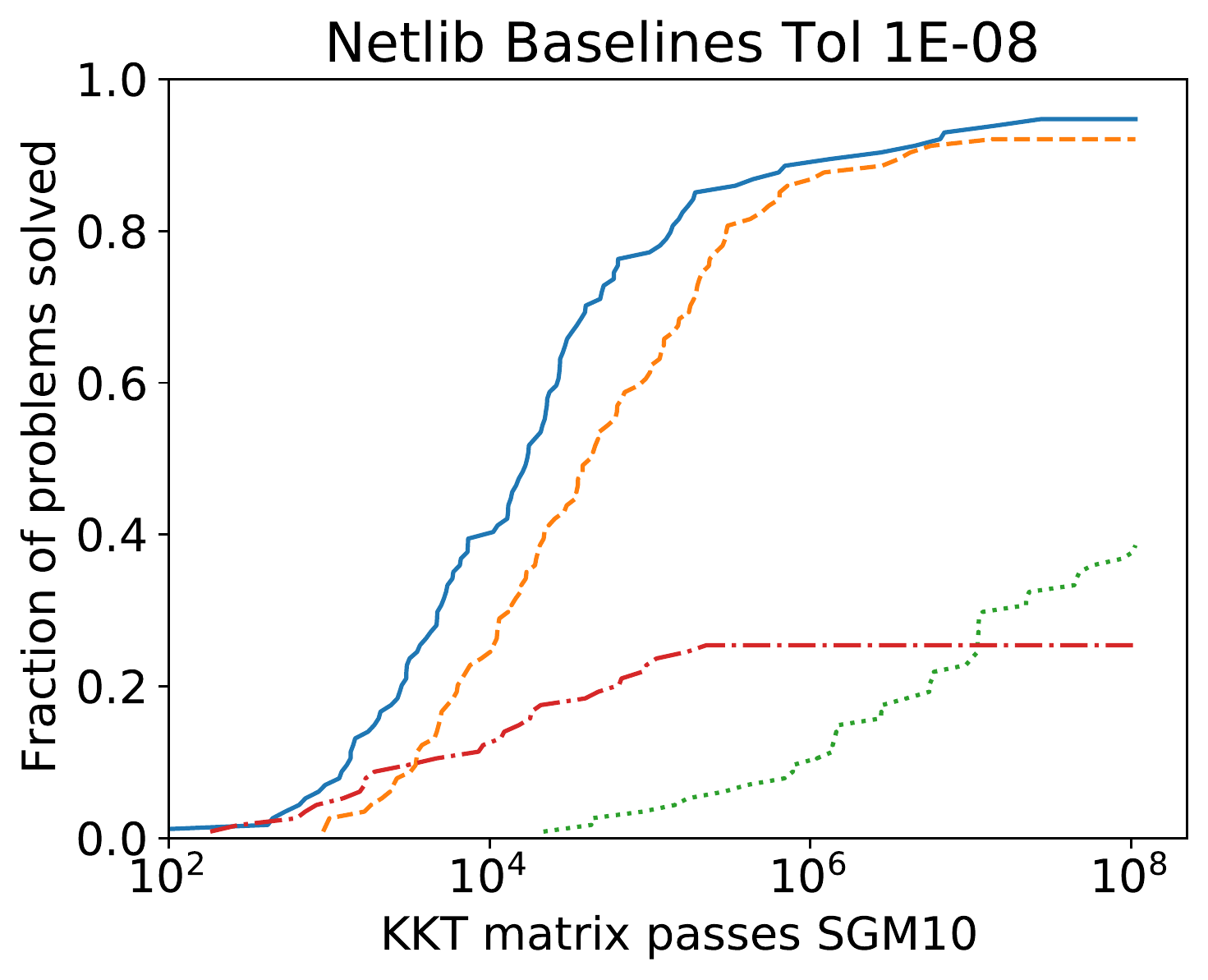}
\end{subfigure}
\begin{subfigure}{.24\textwidth}
  \centering
  \includegraphics[width=.95\linewidth]{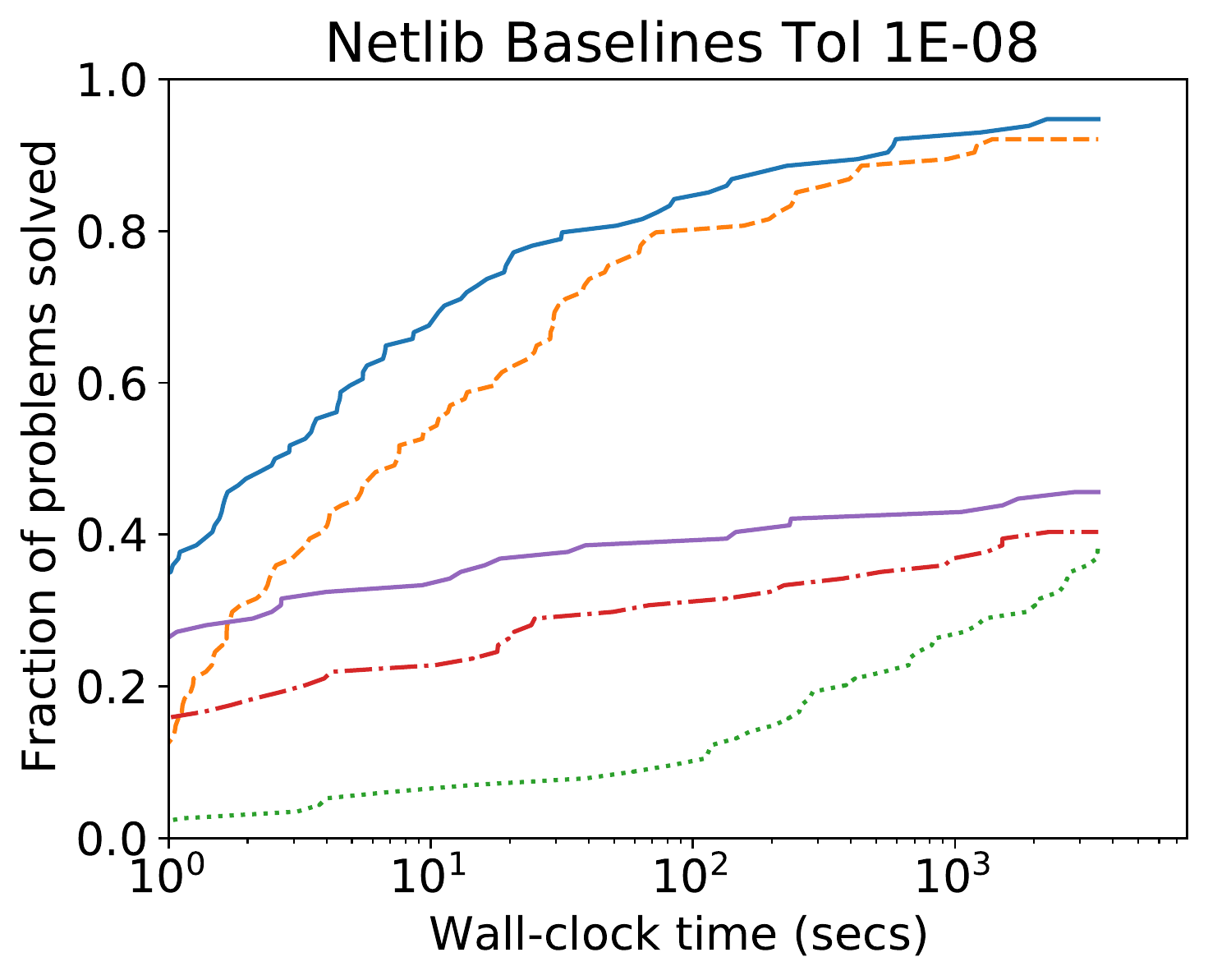}
\end{subfigure}\\
\begin{subfigure}{\textwidth}
\centering
\includegraphics[width=.95\linewidth]{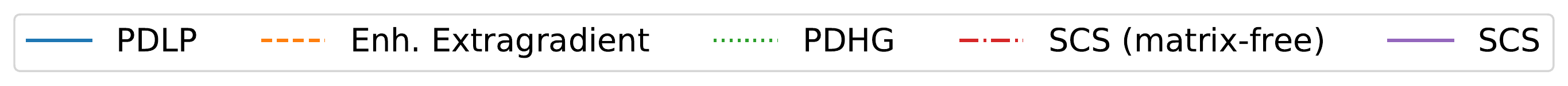}
\end{subfigure}
\caption{Number of problems solved for \texttt{MIP Relaxations} (top), \texttt{LP benchmark} (middle), and \texttt{Netlib} (bottom) datasets.}
\label{fig:pdlp-vs-baseline}
\end{figure}

\subsection{PDLP versus simplex and barrier}\label{sec:medium-gurobi}

In this section, we test the performance of PDLP against the three methods available in Gurobi: barrier, primal simplex, and dual simplex. By default when provided multiple threads, Gurobi runs these three methods concurrently and terminates when the first method completes. 
 We used default termination for Gurobi and set $\epsilon = 10^{-8}$ for PDLP.
We ran experiments with instances from the \texttt{MIP Relaxations} and \texttt{LP benchmark}.
Although, for most instances, Gurobi outperforms PDLP, we found problems for which PDLP exhibits moderate to significant gains. Table~\ref{table:PDLP-close-gurobi} gives examples of instances where our prototype implementation is within a factor of two of the best of the three Gurobi methods.  While further improvements are needed for PDLP to truly compete with the portfolio of methods that Gurobi offers, we interpret these results as evidence that PDLP itself could be of value in this portfolio.

\begin{table}
\centering
\caption{Instances from \texttt{MIP Relaxations} (top) and \texttt{LP benchmark} (bottom) where PDLP is within a factor of 2 of the best of all Gurobi methods. Time to solve in seconds.} 
\label{table:PDLP-close-gurobi}
\begin{tabular}{lcccc}
\toprule
        Instance &  PDLP & Gurobi Barrier  & Gurobi Primal Simp. &Gurobi Dual Simp.   \\
\midrule
ex9 & 1.6 & 102.6 & 181.3 & 47.6 \\
genus-sym-g62-2 & 2.1 & 10.7 & 6.7 & 33.2 \\
highschool1-aigio  & 72.6 & 243.8 & \textgreater 3600 & \textgreater 3600 \\
neos-578379 & 1.4 & 0.7 & 1.7 & 1.8 \\
rwth-timetable & 1870.3 &  \textgreater 3600 &  \textgreater3600 & >3600 \\
\midrule[0.25pt]
ex10 & 4.9 & 63.1 & 16.8 & 7.9 \\
nug08-3rd & 2.2 & 3.2 & 2219.2 & 24.1\\
savsched1 & 35.9 & 25.9 & 56.0 & 261.3\\
\bottomrule
\end{tabular}
\end{table}

\subsection{Large-scale application: PageRank}\label{sec:pagerank}
Nesterov~\cite[equation (7.3)]{nesterov2014pagerank} gives an LP formulation of the standard ``PageRank'' problem. Although the LP formulation is not the best approach to computing PageRank, it is a source of very large instances.
For a random scalable collection of PageRank instances, we used Barabási-Albert~\cite{barabasi1999graph} preferential attachment graphs with approximately three edges per node; see Appendix~\ref{app:large} for details. The results are summarized in Table~\ref{t:pagerank}.
\begin{table}
\centering
\caption{Solve time for PageRank instances. Gurobi barrier has crossover disabled, 1 thread. PDLP and SCS solve to $10^{-8}$ relative accuracy. SCS is matrix-free. Baseline PDHG is unable to solve any instances. Presolve not applied. OOM = Out of Memory. The number of nonzero coefficients per instance is $8 \times (\text{\# nodes}) - 18$.}
\label{t:pagerank}
\begin{tabular}{lccccc}
\toprule
        \# nodes  &  PDLP & SCS & Gurobi Barrier  & Gurobi Primal Simp. &Gurobi Dual Simp.   \\
\midrule
$10^4$  &7.4 sec. & 1.3 sec. & 36 sec. & 37 sec. & 114 sec. \\
$10^5$  &35 sec. & 38 sec. & 7.8 hr. & 9.3 hr. & >24 hr.\\
$10^6$   &11 min. & 25 min. &OOM&  \textgreater 24 hr.& -\\
$10^7$  &5.4 hr. & 3.8 hr. &-&-& -\\
\bottomrule
\end{tabular}
\end{table}

\section{Conclusions and future work}\label{sec:futurework}

We find our experimental results encouraging for the application of FOMs like PDHG to LP. At a minimum, they provide evidence against the claim that FOMs are useful only when moderately accurate solutions are desired. The practical success of our heuristics that lack theoretical guarantees provides fresh motivation for theoreticians to study these methods. It is important, as well, to understand what drives the difficulty of some instances and how they could be transformed to solve more quickly. We hope the community will use the benchmarks and baselines released with this work as a starting point for further investigating new FOMs for LP. With additional algorithmic and implementation refinements, we believe that PDLP or similar approaches could become part of the standard toolkit for linear programming.

\section*{Acknowledgements}
We thank Yura Malitsky for advice on parameter choices for the line search rule of~\cite{malitsky2018linesearch}.
\bibliographystyle{plain}
\bibliography{main}
\appendix

\section{Proof of scale invariance of primal weight initialization scheme}\label{app:scale-invariance}

\begin{proposition}\label{prop:scale-invariance}
Suppose that $\hat{K} = \gamma K$, $\hat{c} = \gamma \alpha_{y} c$, $\hat{q} = \gamma \alpha_{x} q$, $\hat{l} = \alpha_{x} l$, and $\hat{u} = \alpha_{x} u$ for $\alpha_{y}, \alpha_{x}, \gamma \in (0,\infty)$ with $\| c \|_2, \| q \|_2, \| \hat{c} \|_2, \| \hat{q} \|_2 > \epsMach$.
Consider the PDHG algorithm given in \eqref{eq:standard-pdhg} with $\omega = \InitializePrimalWeight{c, q}$.
Let $z^{k}$ be the PDHG iterates on the original problem and $\hat{z}^k$ be the PDHG iterates on the scaled problem with $\hat{x}^{0} = \alpha_{x} x^{0}$ and $\hat{y}^{0} = \alpha_{y} y^{0}$, then:
$\hat{x}^{k} = \alpha_{x} x^{k}, \hat{y}^{k} = \alpha_{y} y^{k}$ for all $k \in \{ 0 \} \cup \N$.
\end{proposition}
\begin{proof}
We will prove this by induction. By definition the result holds for $k = 0$. Define $\hat{\eta} = \eta / \gamma$, and $\hat{\omega} = \| \hat{c} \| / \| \hat{q} \|_2 = \omega (\alpha_y / \alpha_x)$. Then,
\begin{flalign*}
\hat{x}^{k+1} &= \proj_{\hat{X}}\left( \hat{x}^{k} - \hat{\eta} / \hat{\omega} (\hat{c} - \hat{K}^\T \hat{y}^{k} ) \right) \\
&= \proj_{\hat{X}}\left( \alpha_{x} x^{k} - \alpha_{x} \eta / \omega (c - K^\T y^{k} ) \right) \\
&= \alpha_{x} \proj_{X}\left( x^{k} - \eta / \omega (c - K^\T y^{k} ) \right) \\
&= \alpha_x x^{k+1}.
\end{flalign*}
Similarly,
\begin{flalign*}
\hat{y}^{k+1} &= \proj_{\hat{Y}}\left( \hat{y}^{k} - \hat{\eta} \hat{\omega} (\hat{q} - \hat{K} (2 \hat{x}^{k+1} - \hat{x}^k)) \right) \\
&=  \proj_{\hat{Y}}\left( \alpha_y y^{k} -  \alpha_y  \eta \omega (q - K (2 x^{k+1} - x^k)) \right)\\
&= \alpha_y \proj_{Y}\left( \alpha_y y^{k} - \alpha_y  \eta \omega (q - K (2 x^{k+1} - x^k)) \right)\\
&= \alpha_y y^{k+1}.
\end{flalign*}
\end{proof}

\section{\texttt{MIP Relaxations} dataset}\label{sec:currated-miplib-dataset}

MIPLIB 2017~\cite{gleixner2021miplib} is a collection of mixed integer programming (MIP) problems used primarily for developing and benchmarking MIP solvers. MIPLIB contains both a larger collection set (1056 instances) and a smaller benchmark set (240 instances). We select 383 instances from the collection set that satisfy the following criteria:
\begin{itemize}
\item Not tagged as numerically unstable
\item Not tagged as infeasible
\item Not tagged as having indicator constraints
\item Finite optimal objective (if known)
\item The constraint matrix has between $100,000$ and $10,000,000$ nonzero coefficients.
\end{itemize}
For comparison, the MIPLIB benchmark set excludes instances whose constraint matrix has more than $1,000,000$ nonzero coefficients. The upper limit of $10,000,000$ was chosen for the convenience of running experiments. Our set both excludes small instances that may be in the benchmark set and includes instances deemed too large for the benchmark set. From each MIP instance we derive an LP instance by removing the integrality constraints.

\section{Ablation study}\label{sec:ablation}
To study the impact of PDLP's improvements over baseline PDHG, we performed an ablation study, in which we evaluate the consequences of disabling each enhancement separately and evaluate alternative choices. All experiments in this section are performed on the \texttt{MIP Relaxations} dataset. Each of these experiments is run with a limit of 100,000 KKT passes and 6 hours. If the instance is unsolved, the KKT passes are set to 100,000, and the solve time to 6 hours.

\subsection{Step size choice}\label{sec:step-size-choice}
\begin{figure}
\centering
\includegraphics[width=0.45\linewidth]{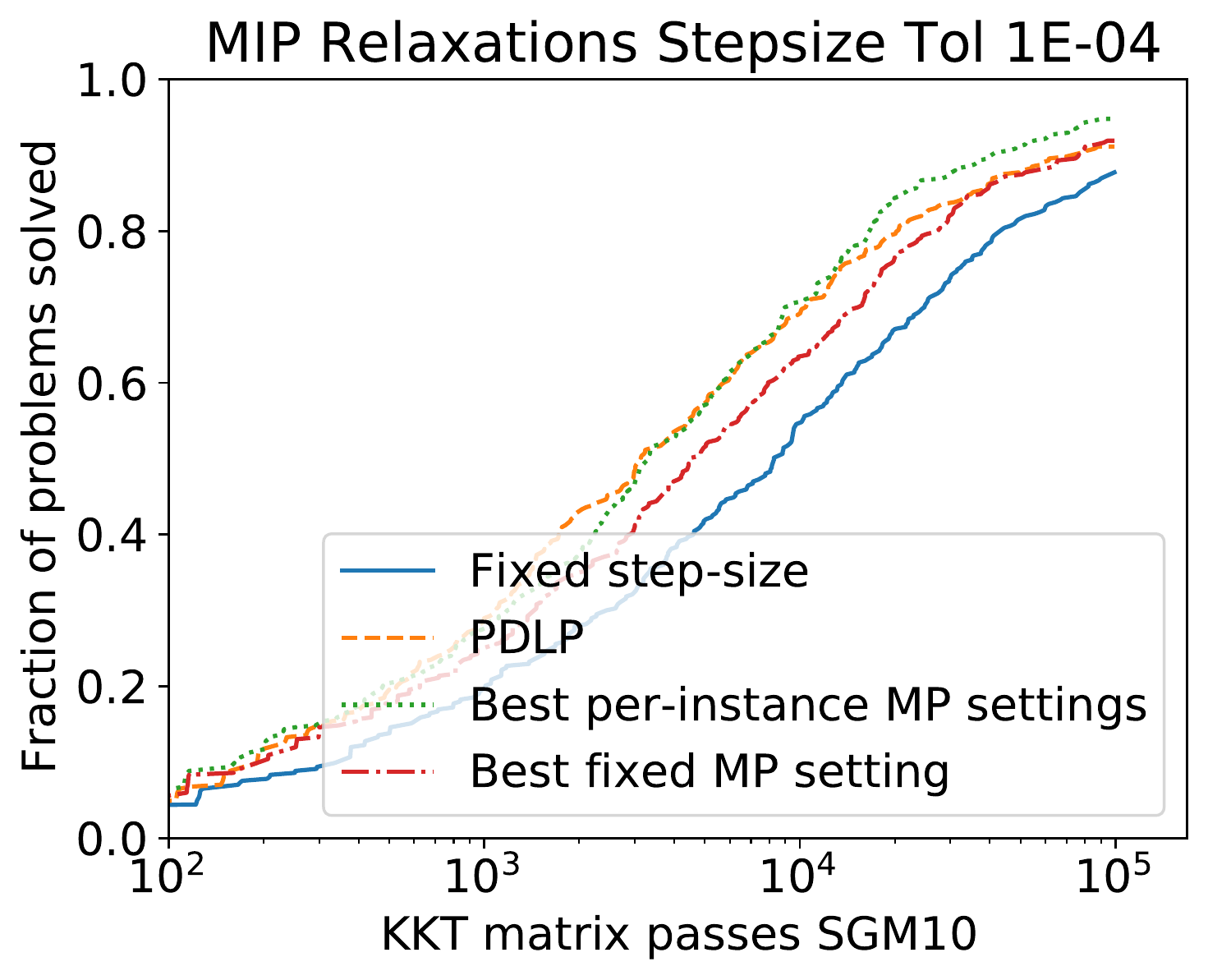}
\includegraphics[width=0.45\linewidth]{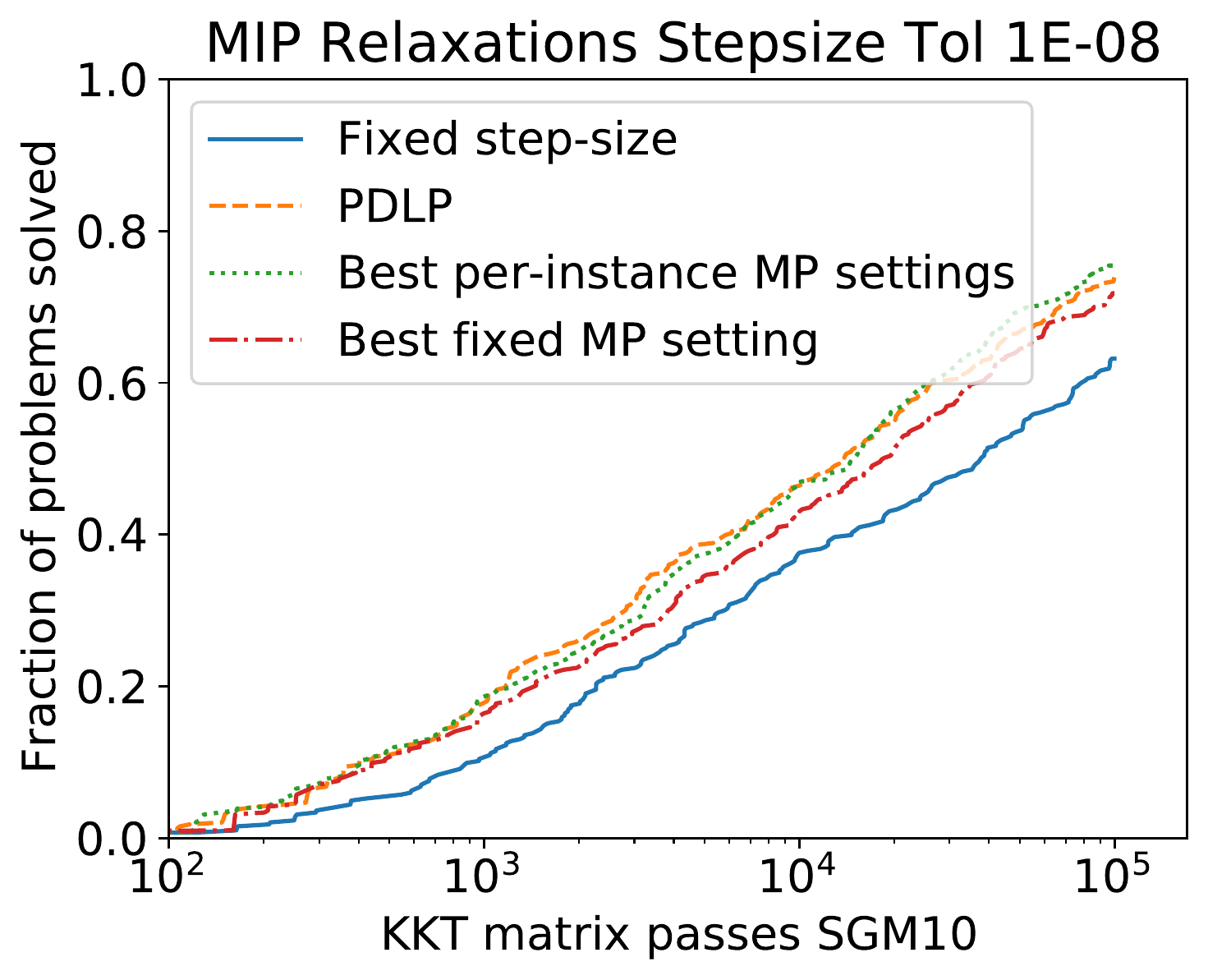}
\caption{Step size ablation experiments on \texttt{MIP Relaxations}}
\label{fig:stepsize}
\end{figure}
\begin{table}
\small
\centering
\caption{\small Performance statistics: MIP Relaxations Stepsize Tol 1E-04}
\label{t:solved-probs-mip_relaxations_stepsize_tol_1E-04}
\begin{tabular}{lccc}
\toprule
                    Experiment &  Solved count &  KKT passes SGM10 &  Solve time secs SGM10 \\
\midrule
               Fixed step-size &           336 &            6207.9 &                   91.3 \\
                          PDLP &           349 &            3058.3 &                   51.1 \\
         Best fixed MP setting &           352 &            3855.6 &                   58.5 \\
 Best per-instance MP settings &           363 &            2869.7 &                   41.4 \\
\bottomrule
\end{tabular}
\end{table}

\begin{table}
\small
\centering
\caption{\small Performance statistics: MIP Relaxations Stepsize Tol 1E-08}
\label{t:solved-probs-mip_relaxations_stepsize_tol_1E-08}
\begin{tabular}{lccc}
\toprule
                    Experiment &  Solved count &  KKT passes SGM10 &  Solve time secs SGM10 \\
\midrule
               Fixed step-size &           242 &           17339.5 &                  469.9 \\
         Best fixed MP setting &           275 &           11660.4 &                  260.1 \\
                          PDLP &           283 &            9773.3 &                  216.0 \\
 Best per-instance MP settings &           289 &            9778.7 &                  193.8 \\
\bottomrule
\end{tabular}
\end{table}

We compare PDLP's adaptive step size rule against three alternatives:
\begin{itemize}
    \item ``Fixed step size'': (baseline PDHG) The step size $\eta$ is fixed to $\eta = 0.9 / \| K \|_2$ where $\| K \|_2$ is estimated via power iteration,
    \item ``Best fixed Malitsky-Pock (MP) setting'': Malitsky and Pock \cite{malitsky2018linesearch}, tuning the hyperparameters via a hyperparameter search, and
    \item ``Best per-instance Malitsky-Pock (MP) setting'': Malitsky and Pock \cite{malitsky2018linesearch}, choosing the best hyperparameters separately for each instance. This is a ``virtual'' solver that combines 42 hyperparameter configurations.
\end{itemize}
The results, in Figure~\ref{fig:stepsize} and Tables~\ref{t:solved-probs-mip_relaxations_stepsize_tol_1E-04} and \ref{t:solved-probs-mip_relaxations_stepsize_tol_1E-08}, show that PDLP is slightly better than tuned Malitsky-Pock, and at high accuracy, almost as good as per-instance tuned Malitsky-pock.

\paragraph{Description of Malitsky and Pock hyperparameters.} Our implementation depends on three hyperparameters: \texttt{breaking\_factor}, \texttt{downscaling\_factor}, and  \texttt{interpolation\_coefficient}. We explain the role of each one by summarizing the linesearch rule. Suppose the algorithm finished iteration $k$ and it does not execute a restart. Thus, the primal weight doesn't not change $\omega_k = \omega_{k+1}$. Mimicking the notation in \cite{malitsky2018linesearch} we define:
$$\theta_k = \frac{\eta_{k-1}}{\eta_k}.$$
Then, the algorithm does the following at iteration $k+1$:
\begin{enumerate}
    \item Update primal iterate $x^{k+1} \leftarrow \proj_X\left(x^k - \frac{\eta_k}{\omega_k}\left(c - K^\top y^k\right)\right).$ 
    \item Pick a candidate for the next step size $\widehat{\eta}_{k+1} \in [\eta_k, \sqrt{1+\theta_k} \eta_k]$. By letting
    $$\widehat{\eta}_{k+1} \leftarrow \eta_k + \texttt{interpolation\_coefficient} \cdot \left(\sqrt{1+\theta_k} - 1\right)\eta_k\quad\text{and} \quad \widehat{\theta}_{k+1}\leftarrow \frac{\eta_k}{\widehat{\eta}_{k+1}}.$$
    \item Compute a candidate for next dual iterate $y_{k+1}$:
    $$ \widehat{y}^{k+1} \leftarrow \proj_Y\left(y^{k} + \omega_{k+1}\widehat{\eta}_{k+1} \left(q - K\left(x^{k+1}+\widehat{\theta}_{k+1}(x^{k+1}-x^k)\right)\right)\right). $$
    \item Check if the linesearch is done;
    \begin{enumerate}
        \item[]  \textbf{If } $\widehat{\eta}_k \| K^\top (\widehat{y}^{k+1} - y^k)\| \leq \texttt{breaking\_factor} \cdot \|\widehat{y}^{k+1} - y^k\|$: 
        $$\eta_{k+1}\leftarrow \widehat{\eta}_{k+1}, \quad \theta_{k+1} \leftarrow \widehat{\theta}_{k+1}, \quad\text{and}\quad y^{k+1}\leftarrow \widehat{y}^{k+1}.$$
        \item[] \textbf{Else}: reduce the step size as follows and then \textbf{go to} Step 3: $$\widehat{\eta}_{k+1} \leftarrow \texttt{downscaling\_factor} \cdot \widehat{\eta}_{k+1},\quad\widehat{\theta}_{k+1}\leftarrow \frac{\eta_k}{\widehat{\eta}_{k+1}}. $$
    \end{enumerate}
\end{enumerate}

In our experiments, we fix \texttt{breaking\_factor} = 1 on guidance from the authors of~\cite{malitsky2018linesearch}. We then perform a grid search on $\texttt{downscaling\_factor} \in \{0.4, 0.5, 0.6, 0.7, 0.8, 0.9\}$ and $\texttt{interpolation\_coefficient} \in \{0.4, 0.5, 0.6, 0.7, 0.8, 0.9, 1.0\}$.

The single best configuration (by count of solved instances) for both $\epsilon = 10^{-4}$ and $\epsilon = 10^{-8}$ is $\texttt{downscaling\_factor}=0.5$ and $\texttt{interpolation\_coefficient}=0.4$.

\subsection{Adaptive restarts}\label{sec:ablation-adaptive-restarts}

\begin{figure}
\centering
\includegraphics[width=0.45\linewidth]{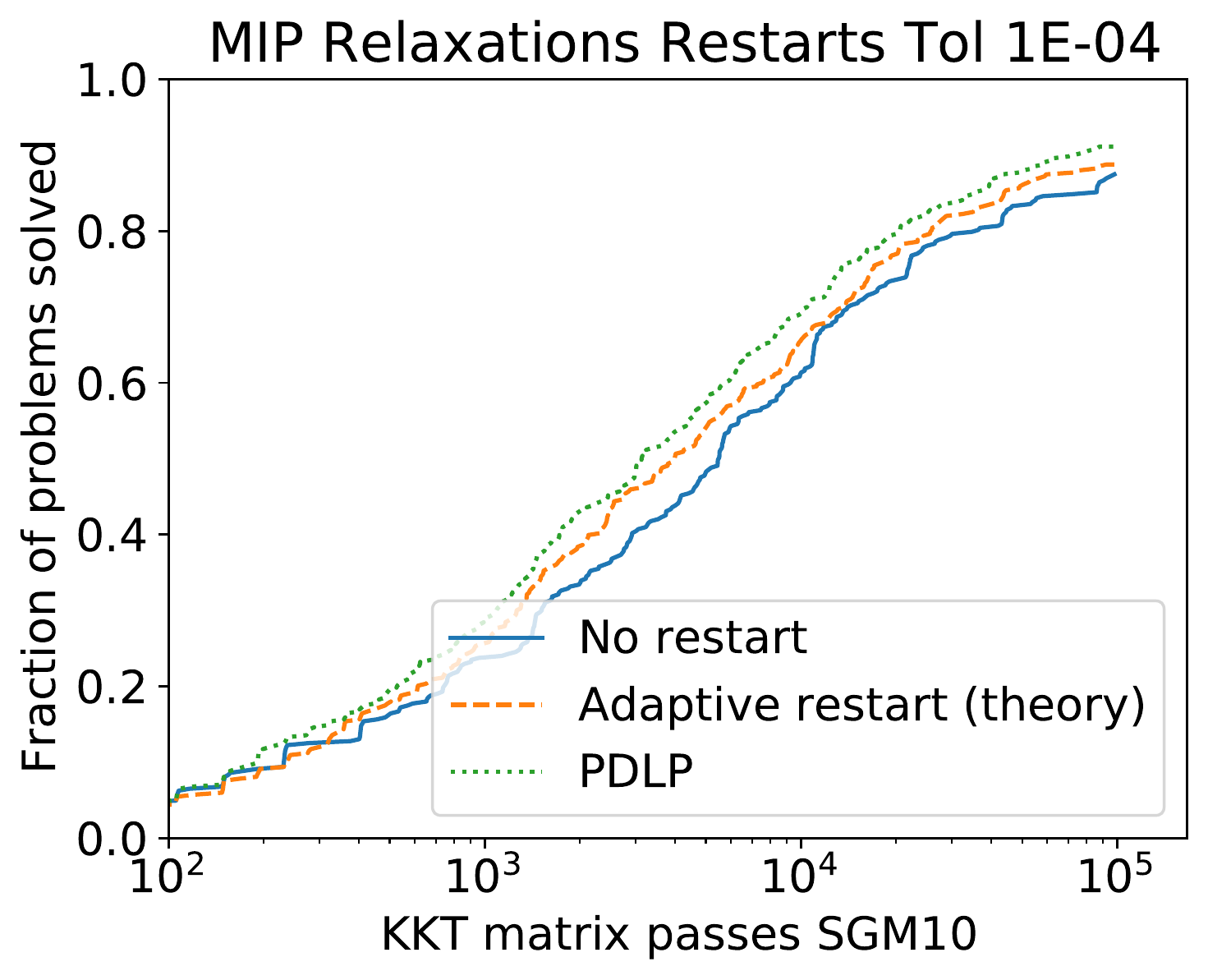}
\includegraphics[width=0.45\linewidth]{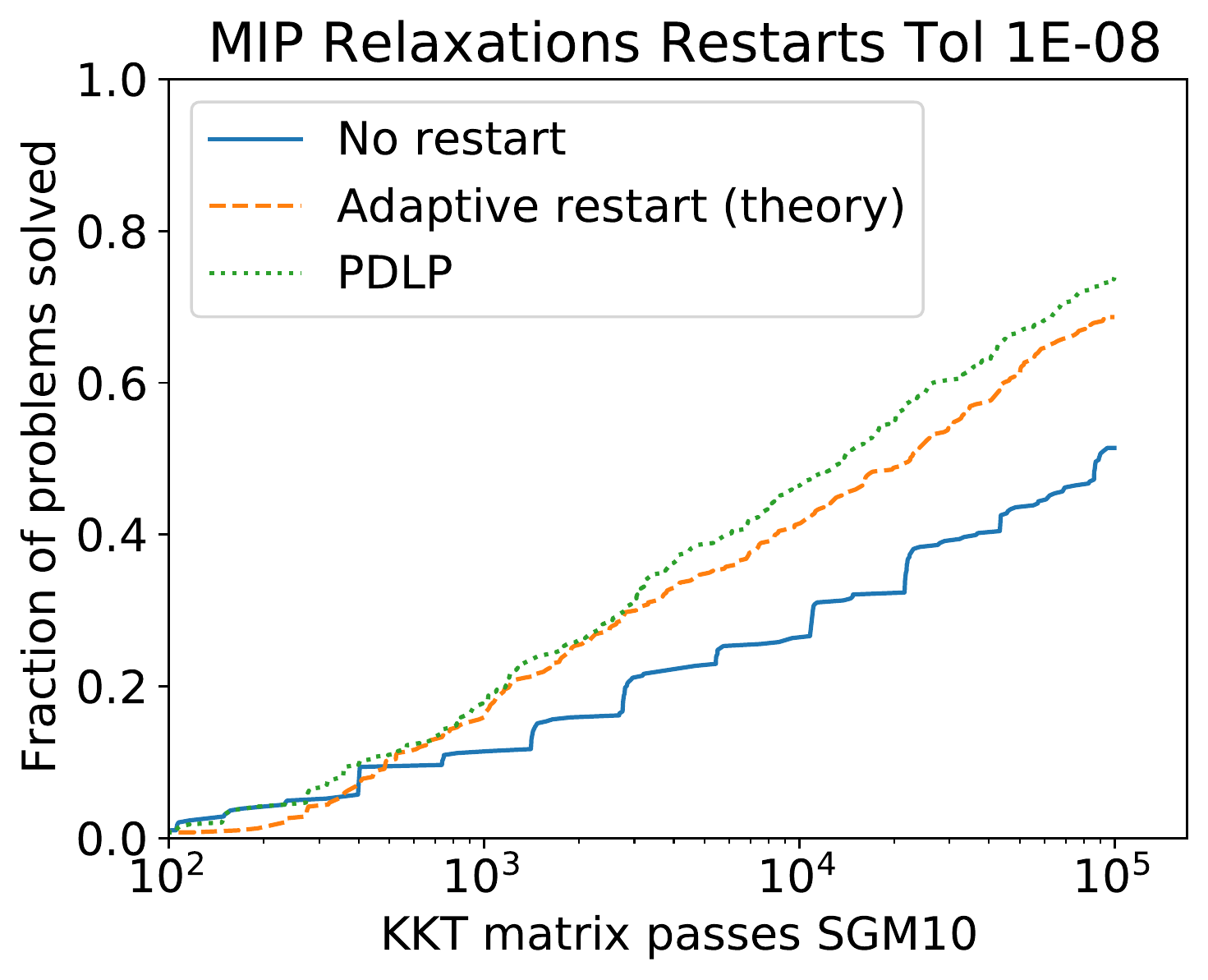}
\caption{Restart ablation experiments on \texttt{MIP Relaxations}}
\label{fig:restarts}
\end{figure}
\begin{table}
\small
\centering
\caption{\small Performance statistics: MIP Relaxations Restarts Tol 1E-04}
\label{t:solved-probs-mip_relaxations_restarts_tol_1E-04}
\begin{tabular}{lccc}
\toprule
                Experiment &  Solved count &  KKT passes SGM10 &  Solve time secs SGM10 \\
\midrule
                No restart &           335 &            4387.9 &                   70.2 \\
 Adaptive restart (theory) &           340 &            3680.5 &                   64.0 \\
                      PDLP &           349 &            3058.3 &                   51.1 \\
\bottomrule
\end{tabular}
\end{table}

\begin{table}
\small
\centering
\caption{\small Performance statistics: MIP Relaxations Restarts Tol 1E-08}
\label{t:solved-probs-mip_relaxations_restarts_tol_1E-08}
\begin{tabular}{lccc}
\toprule
                Experiment &  Solved count &  KKT passes SGM10 &  Solve time secs SGM10 \\
\midrule
                No restart &           197 &           22488.2 &                  960.5 \\
 Adaptive restart (theory) &           263 &           12175.6 &                  308.0 \\
                      PDLP &           283 &            9773.3 &                  215.5 \\
\bottomrule
\end{tabular}
\end{table}

For PDLP, we use
$\beta_{\text{sufficient}} = 0.9$, $\beta_{\text{necessary}} = 0.1$, and $\beta_{\text{artificial}} = 0.5$ as the restart parameters.
For ``Adaptive restart (theory)'' mode we match \cite{applegate2021restarts} and by setting $\beta_{\text{sufficient}} = \beta_{\text{necessary}} = 0.37 = \exp(-1)$. This is equivalent to removing condition (ii) from the restart criteria.
For ``no restart'' mode we disable restarts. For this setting primal weight updates still occur when an artificial restart would have been triggered. In other words, the primal weights are updated on iteration $2, 2^2, 2^3, \dots $.

The performance of adaptive restarts are summarized in Figure~\ref{fig:restarts} and Tables~\ref{t:solved-probs-mip_relaxations_restarts_tol_1E-04} and \ref{t:solved-probs-mip_relaxations_restarts_tol_1E-08}. 
We can see PDLP outperforms ``Adaptive restart (theory)'' mode which in turn beats `no restart'' mode.
This difference is much more pronounced at high accuracy.

\subsection{Primal weight updates}\label{sec:ablation-primal-weight}
\begin{figure}
\centering
\includegraphics[width=0.45\linewidth]{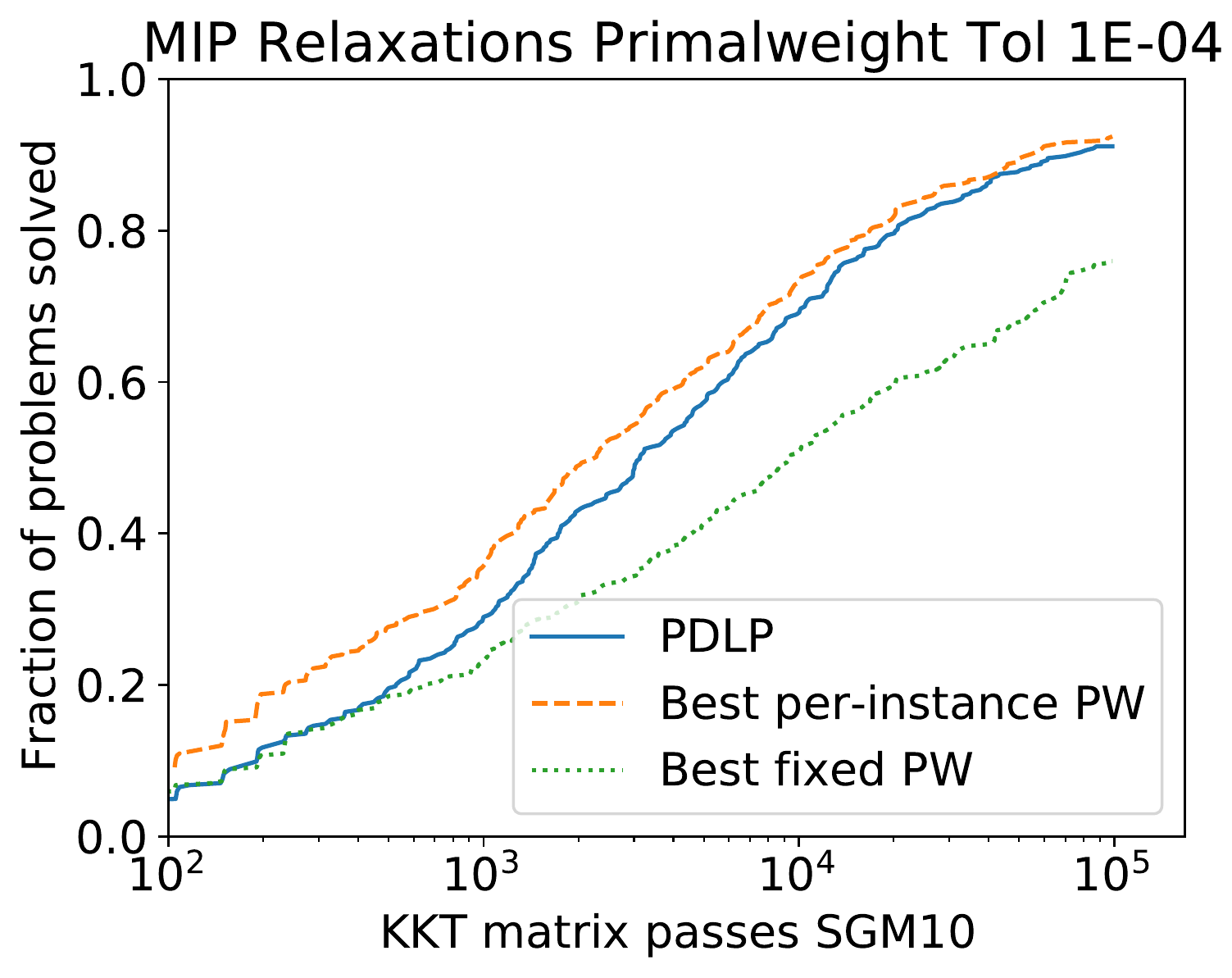}
\includegraphics[width=0.45\linewidth]{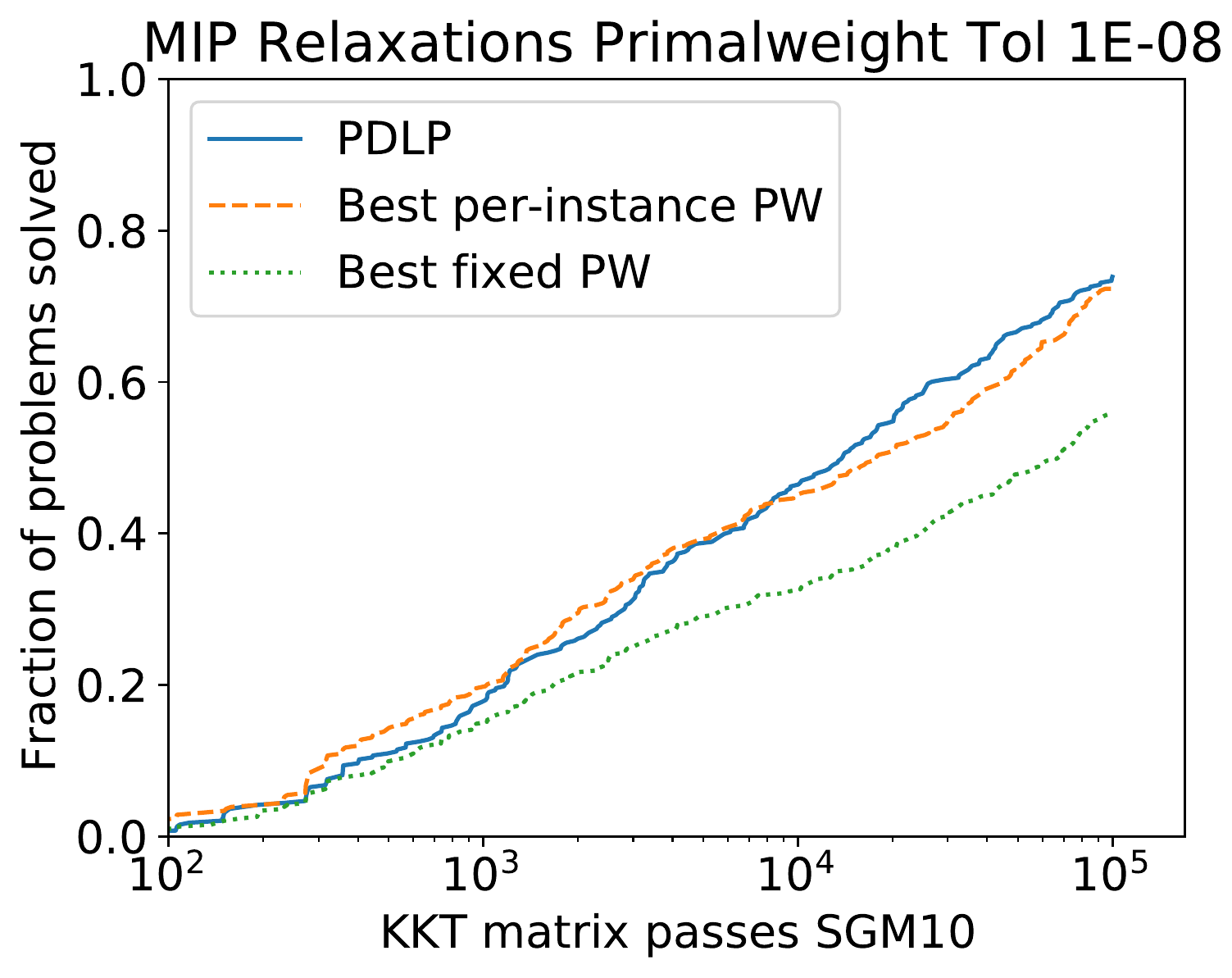}
\caption{Primal weight ablation experiments on \texttt{MIP Relaxations}}
\label{fig:primalweight}
\end{figure}
\begin{table}
\small
\centering
\caption{\small Performance statistics: MIP Relaxations Primalweight Tol 1E-04}
\label{t:solved-probs-mip_relaxations_primalweight_tol_1E-04}
\begin{tabular}{lccc}
\toprule
           Experiment &  Solved count &  KKT passes SGM10 &  Solve time secs SGM10 \\
\midrule
        Best fixed PW &           291 &            6548.0 &                  184.4 \\
                 PDLP &           349 &            3058.3 &                   51.5 \\
 Best per-instance PW &           354 &            2091.3 &                   41.2 \\
\bottomrule
\end{tabular}
\end{table}

\begin{table}
\small
\centering
\caption{\small Performance statistics: MIP Relaxations Primalweight Tol 1E-08}
\label{t:solved-probs-mip_relaxations_primalweight_tol_1E-08}
\begin{tabular}{lccc}
\toprule
           Experiment &  Solved count &  KKT passes SGM10 &  Solve time secs SGM10 \\
\midrule
        Best fixed PW &           214 &           17852.7 &                  707.2 \\
 Best per-instance PW &           277 &            9846.8 &                  246.0 \\
                 PDLP &           283 &            9773.3 &                  216.8 \\
\bottomrule
\end{tabular}
\end{table}

In PDLP, the smoothing parameter is set to $\theta = 0.5$.
As baselines for PDLP's primal weight (PW) updating rule, we compare with using fixed primal weights, setting the primal weight to $\omega = \xi\cdot\InitializePrimalWeight{c, q}$ with the bias $\xi \in \{10^{-5}, 10^{-4}, \ldots, 10^0, \ldots, 10^4, 10^5\}$ chosen by grid search. For these experiments, the smoothing parameter is set to $\theta=0$ to fix the primal weight during the solve.

We compute both the single best value of $\xi$ (by count of solved instances), and the best per-instance value, which defines a ``virtual'' solver. The single best value of $\xi$ is 0.1 at both $\epsilon = 10^{-4}$ and $\epsilon = 10^{-8}$. Qualitatively, the performance of $\xi = 1$, which is a natural default, is very similar to that of $\xi = 0.1$.

From Figure~\ref{fig:primalweight} and Tables~\ref{t:solved-probs-mip_relaxations_primalweight_tol_1E-04} and \ref{t:solved-probs-mip_relaxations_primalweight_tol_1E-08}, we conclude that PDLP is competitive with the best per-instance fixed primal weight at low accuracy, and outperforms it at high accuracy.

\subsection{Presolve}\label{sec:ablation-presolve}
\begin{figure}
\centering
\includegraphics[width=0.45\linewidth]{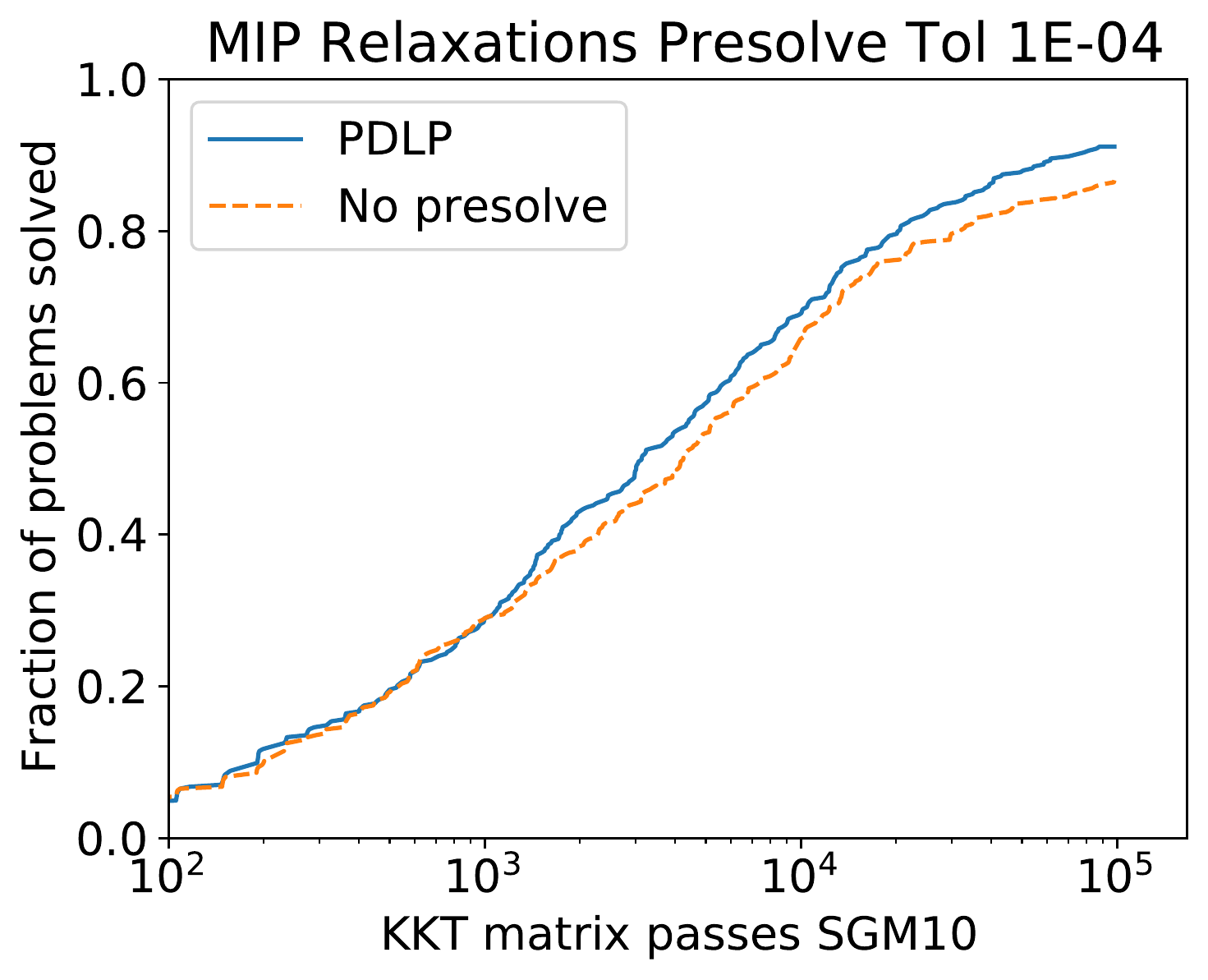}
\includegraphics[width=0.45\linewidth]{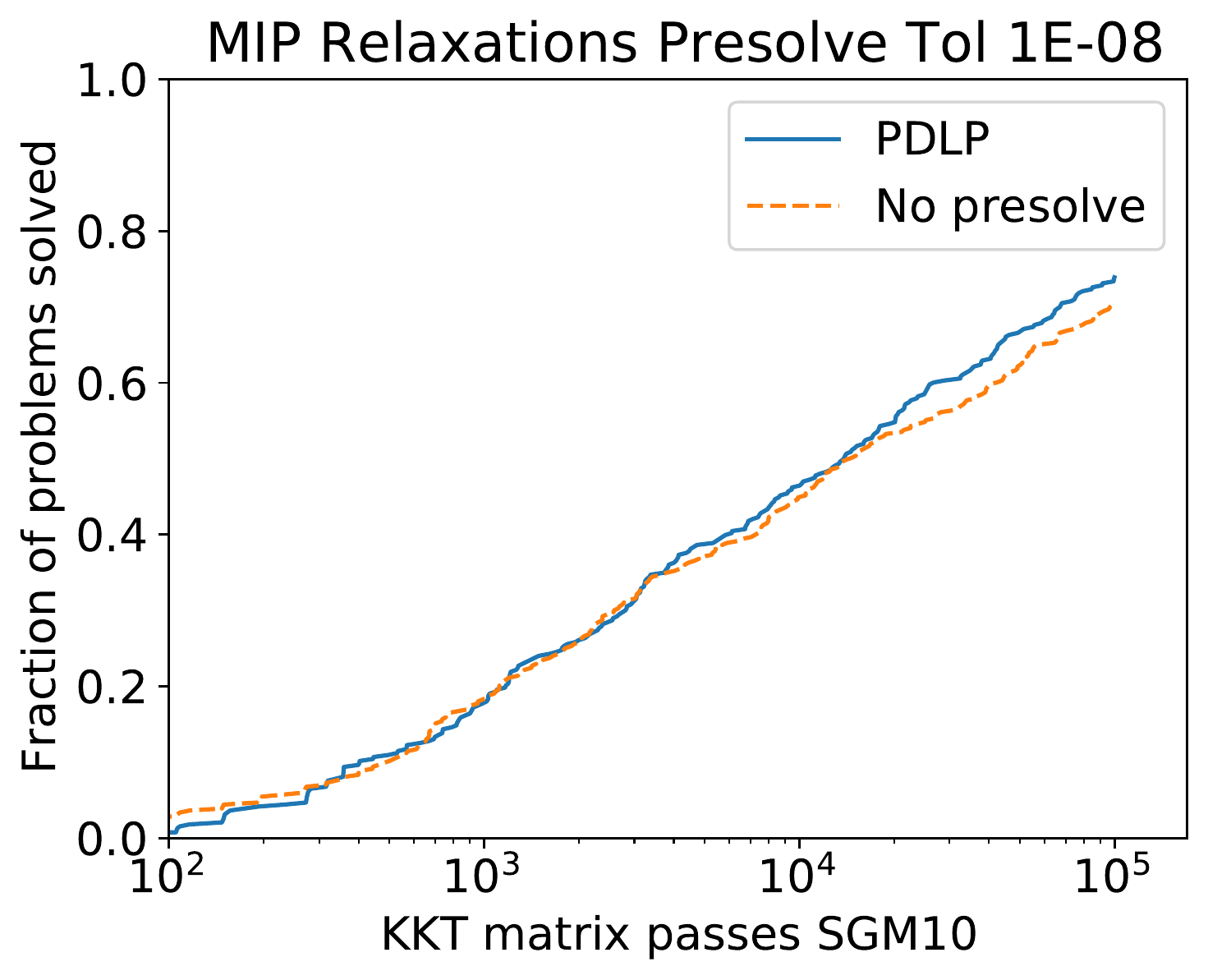}
\caption{Presolve ablation experiments on \texttt{MIP Relaxations}}
\label{fig:presolve}
\end{figure}
\begin{table}
\small
\centering
\caption{\small Performance statistics: MIP Relaxations Presolve Tol 1E-04}
\label{t:solved-probs-mip_relaxations_presolve_tol_1E-04}
\begin{tabular}{lccc}
\toprule
  Experiment &  Solved count &  KKT passes SGM10 &  Solve time secs SGM10 \\
\midrule
 No presolve &           332 &            3615.3 &                   89.9 \\
        PDLP &           349 &            3058.3 &                   51.0 \\
\bottomrule
\end{tabular}
\end{table}

\begin{table}
\small
\centering
\caption{\small Performance statistics: MIP Relaxations Presolve Tol 1E-08}
\label{t:solved-probs-mip_relaxations_presolve_tol_1E-08}
\begin{tabular}{lccc}
\toprule
  Experiment &  Solved count &  KKT passes SGM10 &  Solve time secs SGM10 \\
\midrule
 No presolve &           269 &           10030.6 &                  318.0 \\
        PDLP &           283 &            9773.3 &                  215.8 \\
\bottomrule
\end{tabular}
\end{table}

Figure~\ref{fig:presolve} and Tables~\ref{t:solved-probs-mip_relaxations_presolve_tol_1E-04} and \ref{t:solved-probs-mip_relaxations_presolve_tol_1E-08} measure the impact of presolve. Note that the impact on solve time is greater than the impact on KKT passes, because presolve also makes each KKT pass faster by making the problem smaller.

\subsection{Diagonal preconditioning}\label{sec:app-diagonal}
\begin{figure}
\centering
\includegraphics[width=0.45\linewidth]{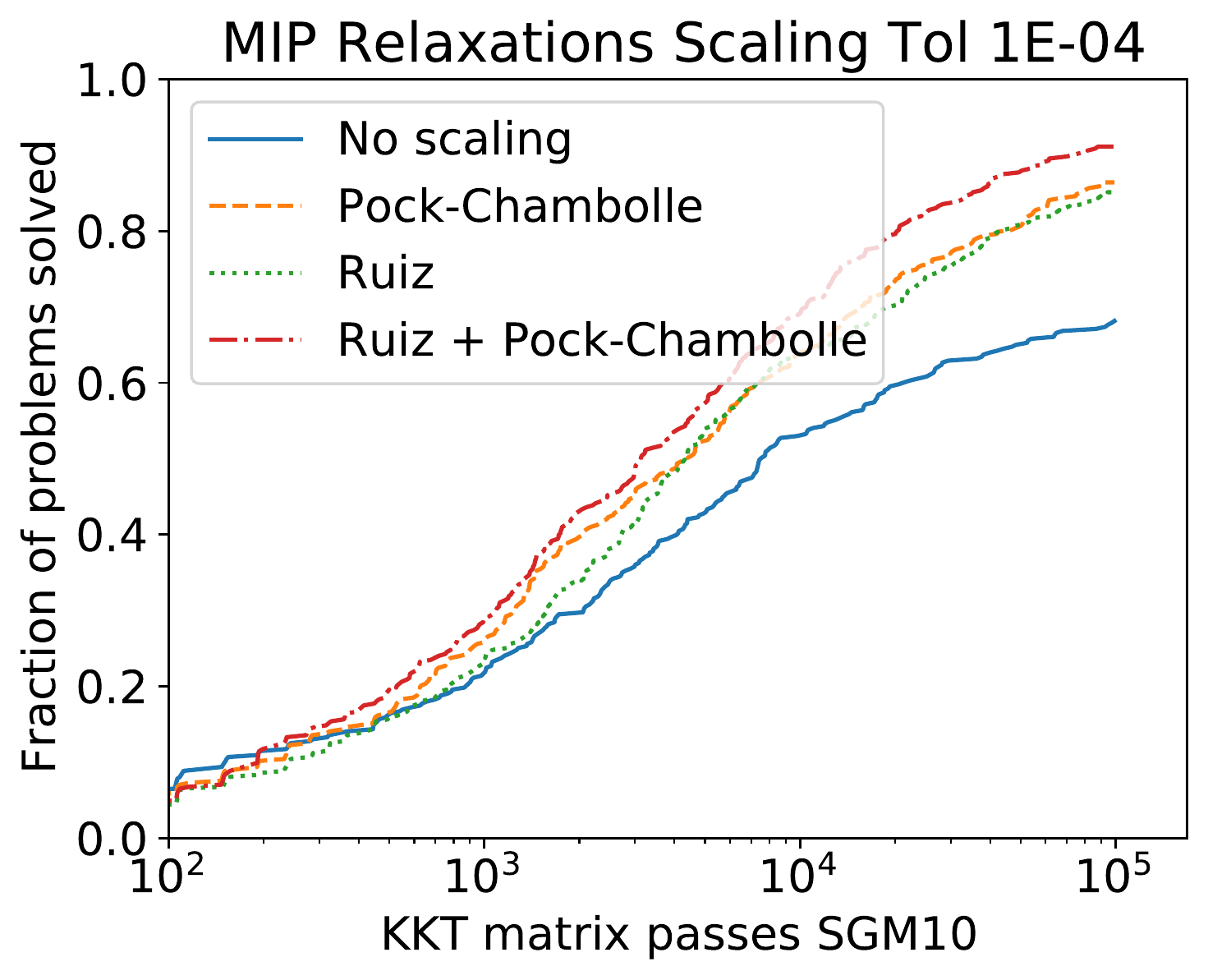}
\includegraphics[width=0.45\linewidth]{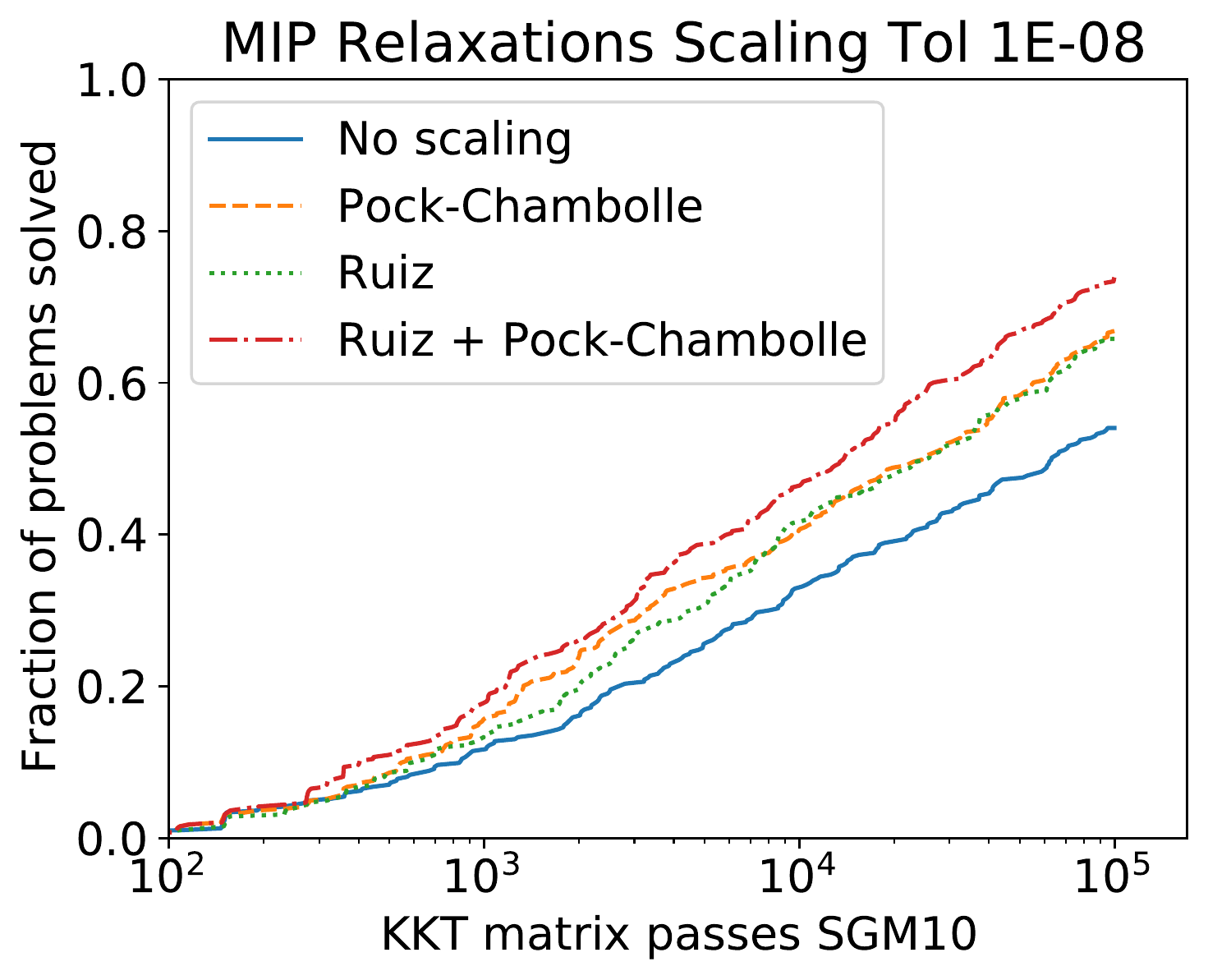}
\caption{Diagonal preconditioning ablation experiments on \texttt{MIP Relaxations}}
\label{fig:scaling}
\end{figure}
\begin{table}
\small
\centering
\caption{\small Performance statistics: MIP Relaxations Scaling Tol 1E-04}
\label{t:solved-probs-mip_relaxations_scaling_tol_1E-04}
\begin{tabular}{lccc}
\toprule
            Experiment &  Solved count &  KKT passes SGM10 &  Solve time secs SGM10 \\
\midrule
            No scaling &           261 &            6700.9 &                  254.1 \\
                  Ruiz &           326 &            4487.7 &                   88.3 \\
        Pock-Chambolle &           331 &            3941.3 &                   78.0 \\
 Ruiz + Pock-Chambolle &           349 &            3058.3 &                   50.9 \\
\bottomrule
\end{tabular}
\end{table}

\begin{table}
\small
\centering
\caption{\small Performance statistics: MIP Relaxations Scaling Tol 1E-08}
\label{t:solved-probs-mip_relaxations_scaling_tol_1E-08}
\begin{tabular}{lccc}
\toprule
            Experiment &  Solved count &  KKT passes SGM10 &  Solve time secs SGM10 \\
\midrule
            No scaling &           207 &           19770.0 &                  787.1 \\
                  Ruiz &           252 &           14028.2 &                  379.9 \\
        Pock-Chambolle &           256 &           12960.0 &                  366.2 \\
 Ruiz + Pock-Chambolle &           283 &            9773.3 &                  218.0 \\
\bottomrule
\end{tabular}
\end{table}

Tables~\ref{t:solved-probs-mip_relaxations_scaling_tol_1E-04} and \ref{t:solved-probs-mip_relaxations_scaling_tol_1E-08} compare the performance of the four diagonal preconditioning techniques as mentioned in Section \ref{sec:diag}. As we can see, the number of solved problems of our proposed preconditioner (Ruiz and Pock-Chambolle) significantly outperform no scaling and the baselines (Pock-Chambolle or Ruiz individually).

Furthermore, Figure~\ref{fig:scaling} shows the number of solved instances as a function of KKT passes for the four different diagonal preconditioners, which further shows a clear separation between PDLP and the baselines.

\section{Additional details on the PageRank LP formulation} \label{app:large}

Based on Nesterov \cite{nesterov2014pagerank}, we formulate the problem of finding a maximal right eigenvector of a stochastic matrix $S$ as a feasible solution of the LP problem:
\begin{eqnarray}\begin{aligned}
\label{eq:pagerank:norm1}
       \mbox{find} \;\; & x\\
\text{subject to:} \;\; & Sx & \leq & & x\\
& \mathbf{1}^\top x & = && 1\\
& x & \geq & & 0 
\end{aligned}
\end{eqnarray}
Nesterov \cite{nesterov2014pagerank} states the constraint $\|x\|_\infty \geq 1$ to enforce $x \neq 0$. We instead use $\mathbf{1}^\top x = 1$ which is equivalent under scaling.

For a random scalable collection of pagerank instances, we used Barabási-Albert~\cite{barabasi1999graph} preferential attachment graphs, using the Julia \texttt{LightGraphs.SimpleGraphs.barabasi\_albert} generator with degree set to 3. We then computed the adjacency matrix and scaled the columns to make the matrix stochastic; call this matrix $S'$. Following the standard PageRank formulation we apply a damping factor to $S'$ and consider $S := \lambda S' + (1-\lambda) J/n$ (where $J = \mathbf{1} \mathbf{1}^\top$ is the all-ones matrix). Intuitively, $S$ encodes a random walk that follows a link in the graph with probability $\lambda$ or jumps to a uniformly random node with probability $1-\lambda$.  

The direct approach to the damping factor results in a completely dense matrix.  Instead we use the fact that $J x = 1$ to rewrite the constraint $Sx \leq x$ in (\ref{eq:pagerank:norm1}) as
\begin{equation}
\lambda (S'x)_i + (1 - \lambda) / n \; \leq \; x_i\;\;\forall i\;\;.
\end{equation}

\section{Additional PDLP improvements results}
\begin{table}
\small
\centering
\caption{\small Performance statistics: MIP Relaxations Improvements Tol 1E-04}
\label{t:solved-probs-mip_relaxations_improvements_tol_1E-04}
\begin{tabular}{lccc}
\toprule
        Experiment &  Solved count &  KKT passes SGM10 &  Solve time secs SGM10 \\
\midrule
              PDHG &           113 &           38958.0 &                 1088.3 \\
         +restarts &           140 &           29739.6 &                  770.4 \\
          +scaling &           221 &           14801.5 &                  313.6 \\
    +primal weight &           315 &            7228.1 &                  110.8 \\
        +step size &           332 &            3615.3 &                   67.6 \\
+presolve (= PDLP) &           349 &            3058.3 &                   42.1 \\
\bottomrule
\end{tabular}
\end{table}

\begin{table}
\small
\centering
\caption{\small Performance statistics: MIP Relaxations Improvements Tol 1E-08}
\label{t:solved-probs-mip_relaxations_improvements_tol_1E-08}
\begin{tabular}{lccc}
\toprule
        Experiment &  Solved count &  KKT passes SGM10 &  Solve time secs SGM10 \\
\midrule
              PDHG &            48 &           68588.8 &                 2232.6 \\
         +restarts &           101 &           47301.1 &                 1284.0 \\
          +scaling &           162 &           25985.7 &                  595.7 \\
    +primal weight &           223 &           18273.4 &                  331.3 \\
        +step size &           269 &           10091.0 &                  181.8 \\
+presolve (= PDLP) &           283 &            9773.3 &                  131.5 \\
\bottomrule
\end{tabular}
\end{table}

\begin{table}
\small
\centering
\caption{\small Performance statistics: LP Benchmark Improvements Tol 1E-04}
\label{t:solved-probs-lp_benchmark_improvements_tol_1E-04}
\begin{tabular}{lccc}
\toprule
        Experiment &  Solved count &  KKT passes SGM10 &  Solve time secs SGM10 \\
\midrule
              PDHG &            10 &           64120.0 &                 2148.5 \\
         +restarts &            10 &           58285.8 &                 2033.9 \\
          +scaling &            17 &           44984.7 &                 1600.3 \\
    +primal weight &            37 &           22232.1 &                  880.1 \\
        +step size &            36 &           13003.7 &                  542.5 \\
+presolve (= PDLP) &            36 &           14721.1 &                  630.7 \\
\bottomrule
\end{tabular}
\end{table}

\begin{table}
\small
\centering
\caption{\small Performance statistics: LP Benchmark Improvements Tol 1E-08}
\label{t:solved-probs-lp_benchmark_improvements_tol_1E-08}
\begin{tabular}{lccc}
\toprule
        Experiment &  Solved count &  KKT passes SGM10 &  Solve time secs SGM10 \\
\midrule
              PDHG &             4 &           87478.5 &                 2784.6 \\
         +restarts &             7 &           69299.7 &                 2210.4 \\
          +scaling &            10 &           58808.8 &                 1929.2 \\
    +primal weight &            14 &           52872.7 &                 1644.8 \\
        +step size &            23 &           38630.3 &                 1336.5 \\
+presolve (= PDLP) &            23 &           35106.0 &                 1281.7 \\
\bottomrule
\end{tabular}
\end{table}

\begin{table}
\small
\centering
\caption{\small Performance statistics: Netlib Improvements Tol 1E-04}
\label{t:solved-probs-netlib_improvements_tol_1E-04}
\begin{tabular}{lccc}
\toprule
        Experiment &  Solved count &  KKT passes SGM10 &  Solve time secs SGM10 \\
\midrule
              PDHG &            14 &           84783.8 &                 1879.2 \\
         +restarts &            15 &           83879.8 &                 1816.4 \\
          +scaling &            43 &           55967.3 &                  485.9 \\
    +primal weight &            94 &           14227.5 &                   30.3 \\
        +step size &            98 &            9443.3 &                   20.4 \\
+presolve (= PDLP) &           103 &            5405.0 &                   11.8 \\
\bottomrule
\end{tabular}
\end{table}

\begin{table}
\small
\centering
\caption{\small Performance statistics: Netlib Improvements Tol 1E-08}
\label{t:solved-probs-netlib_improvements_tol_1E-08}
\begin{tabular}{lccc}
\toprule
        Experiment &  Solved count &  KKT passes SGM10 &  Solve time secs SGM10 \\
\midrule
              PDHG &             4 &           97135.7 &                 2962.2 \\
         +restarts &             8 &           90532.9 &                 2432.7 \\
          +scaling &            22 &           71722.9 &                 1217.9 \\
    +primal weight &            67 &           36843.2 &                  167.1 \\
        +step size &            85 &           23264.1 &                   61.0 \\
+presolve (= PDLP) &            88 &           13419.9 &                   41.2 \\
\bottomrule
\end{tabular}
\end{table}

Tables~\ref{t:solved-probs-mip_relaxations_improvements_tol_1E-04} and \ref{t:solved-probs-mip_relaxations_improvements_tol_1E-08} give a tabular version of the impact of PDLP's improvements on the \texttt{MIP Relaxations} dataset (corresponding to Figure~\ref{fig:miplib-improvements}).  Tables~\ref{t:solved-probs-lp_benchmark_improvements_tol_1E-04} and \ref{t:solved-probs-lp_benchmark_improvements_tol_1E-08} give a tabular version of the impact of PDLP's improvements on the \texttt{LP benchmark} dataset (corresponding to Figure~\ref{fig:lpbenchmark-improvements}).  Tables~\ref{t:solved-probs-netlib_improvements_tol_1E-04} and \ref{t:solved-probs-netlib_improvements_tol_1E-08} give a tabular version of the impact of PDLP's improvements on the \texttt{Netlib} dataset (corresponding to Figure~\ref{fig:netlib-improvements}).

\section{Additional baseline comparison results}
\begin{table}
\small
\centering
\caption{\small Performance statistics: MIP Relaxations Baselines Tol 1E-04}
\label{t:solved-probs-mip_relaxations_baselines_tol_1E-04}
\begin{tabular}{lccc}
\toprule
        Experiment &  Solved count &  KKT passes SGM10 &  Solve time secs SGM10 \\
\midrule
              PDHG &           159 &           45720.5 &                  922.9 \\
 SCS (matrix-free) &           287 &           37027.2 &                  257.0 \\
               SCS &           317 &                 - &                  149.7 \\
Enh. Extragradient &           351 &            6028.7 &                   75.2 \\
              PDLP &           371 &            3236.6 &                   38.4 \\
\bottomrule
\end{tabular}
\end{table}

\begin{table}
\small
\centering
\caption{\small Performance statistics: MIP Relaxations Baselines Tol 1E-08}
\label{t:solved-probs-mip_relaxations_baselines_tol_1E-08}
\begin{tabular}{lccc}
\toprule
        Experiment &  Solved count &  KKT passes SGM10 &  Solve time secs SGM10 \\
\midrule
              PDHG &            80 &           79085.1 &                 2026.7 \\
 SCS (matrix-free) &           124 &           40486.1 &                 1006.9 \\
               SCS &           156 &                 - &                  675.3 \\
Enh. Extragradient &           302 &           21216.9 &                  207.4 \\
              PDLP &           334 &           11381.1 &                  106.4 \\
\bottomrule
\end{tabular}
\end{table}

\begin{table}
\small
\centering
\caption{\small Performance statistics: LP Benchmark Baselines Tol 1E-04}
\label{t:solved-probs-lp_benchmark_baselines_tol_1E-04}
\begin{tabular}{lccc}
\toprule
         Experiment &  Solved count &  KKT passes SGM10 &  Solve time secs SGM10 \\
\midrule
               PDHG &            12 &           67792.0 &                 2009.5 \\
  SCS (matrix-free) &            25 &           51040.9 &                 1118.7 \\
                SCS &            26 &                 - &                 1155.6 \\
 Enh. Extragradient &            27 &           25808.3 &                  944.2 \\
               PDLP &            32 &           16679.4 &                  613.8 \\
\bottomrule
\end{tabular}
\end{table}

\begin{table}
\small
\centering
\caption{\small Performance statistics: LP Benchmark Baselines Tol 1E-08}
\label{t:solved-probs-lp_benchmark_baselines_tol_1E-08}
\begin{tabular}{lccc}
\toprule
        Experiment &  Solved count &  KKT passes SGM10 &  Solve time secs SGM10 \\
\midrule
              PDHG &             6 &           92556.5 &                 2693.1 \\
 SCS (matrix-free) &             7 &           63771.2 &                 2155.6 \\
               SCS &             9 &                 - &                 2017.1 \\
Enh. Extragradient &            13 &           54795.9 &                 1693.4 \\
              PDLP &            22 &           37937.0 &                 1213.4 \\
\bottomrule
\end{tabular}
\end{table}

\begin{table}
\small
\centering
\caption{\small Performance statistics: Netlib Baselines Tol 1E-04}
\label{t:solved-probs-netlib_baselines_tol_1E-04}
\begin{tabular}{lccc}
\toprule
        Experiment &  Solved count &  KKT passes SGM10 &  Solve time secs SGM10 \\
\midrule
              PDHG &            63 &          360059.0 &                  558.2 \\
Enh. Extragradient &           110 &           15722.3 &                   13.2 \\
 SCS (matrix-free) &           110 &           26134.0 &                   15.4 \\
              PDLP &           113 &            6708.7 &                    6.9 \\
               SCS &           113 &                 - &                   11.5 \\
\bottomrule
\end{tabular}
\end{table}

\begin{table}
\small
\centering
\caption{\small Performance statistics: Netlib Baselines Tol 1E-08}
\label{t:solved-probs-netlib_baselines_tol_1E-08}
\begin{tabular}{lccc}
\toprule
        Experiment &  Solved count &  KKT passes SGM10 &  Solve time secs SGM10 \\
\midrule
              PDHG &            44 &          391013.9 &                 1376.2 \\
 SCS (matrix-free) &            46 &           47943.8 &                  559.1 \\
               SCS &            52 &                 - &                  345.0 \\
Enh. Extragradient &           105 &           42993.3 &                   30.6 \\
              PDLP &           108 &           18866.1 &                   18.2 \\
\bottomrule
\end{tabular}
\end{table}

Tables~\ref{t:solved-probs-mip_relaxations_baselines_tol_1E-04} and \ref{t:solved-probs-mip_relaxations_baselines_tol_1E-08} give a tabular version of the comparison of PDLP with other first-order baselines on the \texttt{MIP Relaxations} dataset,
Tables~\ref{t:solved-probs-lp_benchmark_baselines_tol_1E-04} and \ref{t:solved-probs-lp_benchmark_baselines_tol_1E-08} give a tabular version of the comparison of PDLP with other first-order baselines on the \texttt{LP benchmark} dataset, and
Tables~\ref{t:solved-probs-netlib_baselines_tol_1E-04} and \ref{t:solved-probs-netlib_baselines_tol_1E-08} give a tabular version of the comparison of PDLP with other first-order baselines on the \texttt{Netlib} dataset (corresponding to Figure~\ref{fig:pdlp-vs-baseline}). Each of these experiments is run with a time limit of 1 hour. If the instance is unsolved, the KKT passes are set to 100,000, and the solve time to 1 hour.

\section{Instructions for reproducing experiments}

This section documents the precise commands and command-line arguments for each experiment in the paper. These instructions are supplemental to the READMEs in the \texttt{FirstOrderLp} code directory (\url{https://github.com/google-research/FirstOrderLp.jl}). We assume that readers have already followed the instructions in the READMEs to set up and ``instantiate'' the Julia environment, and collect or generate all the datasets. Examples assume that the current working directory is \texttt{FirstOrderLp}.

The full suite of experiments takes approximately $12,500$ CPU-hours to run, and hence requires use of a cluster or cloud computing environment. Given the idiosyncrasies of these environments, we do not provide additional utilities for distributing the experiments. See Section~\ref{sec:setup} for details on the computing platforms we used.

The following base invocations show how to run the two main scripts without any custom arguments.

\begin{lstlisting}[caption={\texttt{solve\_qp.jl} base invocation},label=code:base]
  julia --project=scripts scripts/solve_qp.jl
        --instance_path=$INSTANCE
        --output_dir=$OUTPUT_DIR
\end{lstlisting}

\begin{lstlisting}[caption={\texttt{solve\_lp\_external.jl} base invocation},label=code:base_external]
  julia --project=scripts scripts/solve_lp_external.jl
        --instance_path=$INSTANCE
        --output_dir=$OUTPUT_DIR
\end{lstlisting}

\texttt{solve\_qp.jl} runs methods implemented in the \texttt{FirstOrderLp} module. \texttt{solve\_lp\_external.jl} runs external solvers (specifically, SCS). 

\subsection{Benchmark collection}

The commands used to collect the \texttt{MIP Relaxations}, \texttt{LP Benchmark}, and \texttt{Netlib} benchmarks are described in the \texttt{benchmarking} subdirectory of the \texttt{FirstOrderLp} code directory. \texttt{README.md} provides more detailed instructions, and \texttt{collect\_mip\_relaxations.sh}, \texttt{collect\_lp\_benchmark.sh}, and \texttt{collect\_netlib\_benchmark.sh} give illustrative scripts for collecting the benchmarks.

\subsection{Tolerances}\label{sec:codetolerance}

In experiments we often solve at termination tolerances $\epsilon = 10^{-4}$ and $\epsilon = 10^{-8}$. The following command-line arguments to \texttt{solve\_qp.jl} are used to set these tolerances.

\begin{lstlisting}[caption={\texttt{solve\_qp.jl} arguments for $\epsilon = 10^{-4}$}]
--relative_optimality_tol 1e-4 --absolute_optimality_tol 1e-4
\end{lstlisting}

\begin{lstlisting}[caption={\texttt{solve\_qp.jl} arguments for $\epsilon = 10^{-8}$}]
--relative_optimality_tol 1e-8 --absolute_optimality_tol 1e-8
\end{lstlisting}

\subsection{Improvements experiment}

This section documents the command-line settings corresponding to the experiments in Section~\ref{sec:improvement-measured-impact} that measure the impact of PDLP's improvements over baseline PDHG.

The following common settings apply for each run:

\begin{lstlisting}[caption={Common settings for each run},label=code:improvementscommon]
--kkt_matrix_pass_limit=100000 --restart_to_current_metric=gap_over_distance --verbosity=0 --method=pdhg
\end{lstlisting}

For each solve, use the base invocation for \texttt{solve\_qp.jl} (Listing~\ref{code:base}), a \textcolor{blue}{tolerance setting} (Section~\ref{sec:codetolerance}), \textcolor{OliveGreen}{common settings} (Listing~\ref{code:improvementscommon}), and \textcolor{red}{one set of parameters} below. See the documentation in READMEs and source code for how to set \texttt{\$OUTPUT\_DIR} and process the results.

For example, the following command solves the ``+ scaling'' setting with $\epsilon = 10^{-4}$:
\begin{Verbatim}[commandchars=\\\{\}]
  julia --project=scripts scripts/solve_qp.jl
        --instance_path=$INSTANCE
        --output_dir=$OUTPUT_DIR
\textcolor{blue}{        --relative_optimality_tol=1e-4}
\textcolor{blue}{        --absolute_optimality_tol=1e-4}
\textcolor{OliveGreen}{        --kkt_matrix_pass_limit=100000}
\textcolor{OliveGreen}{        --restart_to_current_metric=gap_over_distance}
\textcolor{OliveGreen}{        --verbosity=0 --method=pdhg}
\textcolor{red}{        --step_size_policy=constant}
\textcolor{red}{        --primal_weight_update_smoothing=0.0}
\textcolor{red}{        --scale_invariant_initial_primal_weight=false}
\end{Verbatim}

Parameter settings:

\begin{enumerate}
    \item ``PDHG'': (on original un-presolved dataset)
    \begin{lstlisting}
--step_size_policy=constant --l_inf_ruiz_iterations=0 --pock_chambolle_rescaling=false --l2_norm_rescaling=false --restart_scheme=no_restart --primal_weight_update_smoothing=0.0 --scale_invariant_initial_primal_weight=false
    \end{lstlisting}
    \item ``+ restarts'':
    \begin{lstlisting}
--step_size_policy=constant --l_inf_ruiz_iterations=0 --pock_chambolle_rescaling=false --l2_norm_rescaling=false --primal_weight_update_smoothing=0.0 --scale_invariant_initial_primal_weight=false
    \end{lstlisting}
    \item ``+ scaling'':
    \begin{lstlisting}
--step_size_policy=constant --primal_weight_update_smoothing=0.0 --scale_invariant_initial_primal_weight=false
    \end{lstlisting}
    \item ``+primal weight'':
    \begin{lstlisting}
--step_size_policy=constant
    \end{lstlisting}
    \item ``+step size'': No additional parameters
    \item ``+presolve (= PDLP)'': Switch to presolved dataset.
\end{enumerate}

\subsection{Comparison with other first-order baselines}

This section documents the command-line settings corresponding to the experiments in Section~\ref{sec:baselines} that compare PDLP with SCS and enhanced Extragradient.

\subsubsection{SCS (\texttt{solve\_lp\_external.jl})}

SCS is invoked via \texttt{solve\_lp\_external.jl}. The following common settings apply for all SCS runs:

\begin{lstlisting}[caption={Common settings for SCS runs}]
--scs-normalize=true --iteration_limit=1000000000
\end{lstlisting}

Because SCS does not support time limits, we use the \texttt{timeout} command to stop SCS after one hour. For example:

\begin{lstlisting}
timeout 1h julia --project=scripts scripts/solve_lp_external.jl --solver=scs-direct ...
\end{lstlisting}

The following arguments\footnote{In preliminary experiments on the MIP relaxations dataset, SCS performed better at $10^{-4}$ with this custom setting of acceleration lookback, which disables Anderson Acceleration.} are used to set $\epsilon = 10^{-4}$.

\begin{lstlisting}[caption={SCS arguments for $\epsilon = 10^{-4}$}]
--tolerance=1e-4 --scs-acceleration_lookback=0
\end{lstlisting}

The following arguments are used to set $\epsilon = 10^{-8}$.

\begin{lstlisting}[caption={SCS arguments for $\epsilon = 10^{-8}$}]
--tolerance=1e-8
\end{lstlisting}

The following arguments\footnote{In preliminary experiments on the MIP relaxations dataset, SCS (matrix-free) performed better with \texttt{cg\_rate=1.01}, which controls the rate at which the conjugate gradient convergence tolerance decreases as a function of the iteration number.} select SCS in matrix-free mode:

\begin{lstlisting}[caption={Configuration for SCS (matrix-free)}]
--solver=scs-indirect --scs-cg_rate=1.01
\end{lstlisting}

The following arguments select SCS in its default mode that uses a cached $LDL$ factorization to solve the linear system that arises at each iteration:

\begin{lstlisting}[caption={Configuration for SCS (default)}]
--solver=scs-direct
\end{lstlisting}

\subsubsection{PDLP and Extragradient (\texttt{solve\_qp.jl})}

PDLP and enhanced Extragradient are invoked via \texttt{solve\_qp.jl}.

The following common settings apply to both PDLP and Extragradient.

\begin{lstlisting}[caption={Common settings for PDLP and Extragradient}]
--time_sec_limit=3600 --restart_to_current_metric=gap_over_distance --verbosity=0
\end{lstlisting}

The following two settings select either the PDLP or enhanced Extragradient methods.

\begin{lstlisting}[caption={Configuration for PDLP}]
--method=pdhg
\end{lstlisting}

\begin{lstlisting}[caption={Configuration for enhanced Extragradient}]
--method=mirror-prox
\end{lstlisting}

\subsection{PDLP versus simplex and barrier}

This section lists the commands corresponding to the experiments in Section~\ref{sec:medium-gurobi} that compare PDLP with Gurobi's simplex and barrier algorithms.

\begin{lstlisting}[caption={Command for Gurobi Barrier},label=list:grbbarrier]
gurobi_cl TimeLimit=3600 Method=2 Crossover=0 Threads=1 $INSTANCE
\end{lstlisting}

\begin{lstlisting}[caption={Command for Gurobi Primal Simplex},label=list:grbprimal]
gurobi_cl TimeLimit=3600 Method=0 Threads=1 $INSTANCE
\end{lstlisting}

\begin{lstlisting}[caption={Command for Gurobi Dual Simplex},label=list:grbdual]
gurobi_cl TimeLimit=3600 Method=1 Threads=1 $INSTANCE
\end{lstlisting}

\begin{lstlisting}[caption={Command for PDLP},label=list:pdlp1h]
julia --project=scripts scripts/solve_qp.jl
      --instance_path=$INSTANCE
      --output_dir=$OUTPUT_DIR
      --relative_optimality_tol=1e-8
      --absolute_optimality_tol=1e-8
      --time_sec_limit=3600
      --restart_to_current_metric=gap_over_distance
      --verbosity=0
      --method=pdhg
\end{lstlisting}

\subsection{Large-scale application: PageRank}
This section describes the commands corresponding to the experiments in Section~\ref{sec:pagerank} that compares PDLP, SCS, and Gurobi's methods on PageRank instances. 

The commands for Gurobi methods are the same as in Listings~\ref{list:grbbarrier},~\ref{list:grbprimal}, and~\ref{list:grbdual}. The command for PDLP is the same as in Listing~\ref{list:pdlp1h}. The command for SCS follows:

\begin{lstlisting}[caption={Command for SCS}]
timeout 1h julia --project=scripts scripts/solve_lp_external.jl
      --instance_path=$INSTANCE
      --output_dir=$OUTPUT_DIR
      --scs-normalize=true
      --iteration_limit=1000000000
      --tolerance=1e-8
      --solver=scs-indirect
      --scs-cg_rate=1.01
\end{lstlisting}

\subsection{Ablation study}

In the ablation study, PDLP is invoked as:
\begin{lstlisting}[caption={PDLP configuration for the ablation study},label=list:pdlp_ablation]
  julia --project=scripts scripts/solve_qp.jl
        --instance_path=$INSTANCE
        --output_dir=$OUTPUT_DIR
        --relative_optimality_tol $TOLERANCE
        --absolute_optimality_tol $TOLERANCE
        --method=pdhg
        --restart_to_current_metric=gap_over_distance
        --kkt_matrix_pass_limit=100000
        --verbosity=0
\end{lstlisting}

on the \texttt{MIP Relaxations} dataset (to which presolve has been applied).

\subsection{Step size choice}
This section describes the commands corresponding to the ablation experiments in Section~\ref{sec:step-size-choice} on the step size choice.

The fixed step-size rule is invoked by appending the following argument to the command in Listing~\ref{list:pdlp_ablation}:

\begin{lstlisting}
--step_size_policy=constant
\end{lstlisting}

The Malitsky and Pock step size rule is invoked by appending the following arguments to the command in Listing~\ref{list:pdlp_ablation}:

\begin{lstlisting}
--step_size_policy=malitsky-pock
--malitsky_pock_breaking_factor=1.0
--malitsky_pock_downscaling_factor=$DOWNSCALING_FACTOR
--malitsky_pock_interpolation_coefficient=$INTERPOLATION_COEFFICIENT
\end{lstlisting}

\subsubsection{Adaptive restarts}
This section describes the commands corresponding to the ablation experiments in Section~\ref{sec:ablation-adaptive-restarts} on restarts.

The ``No restart'' setting is invoked by appending the following argument to the command in Listing~\ref{list:pdlp_ablation}:

\begin{lstlisting}
--restart_scheme=no_restart
\end{lstlisting}

The ``Adaptive restart (theory)'' setting is invoked by appending the following arguments to the command in Listing~\ref{list:pdlp_ablation}:

\begin{lstlisting}
--restart_to_current_metric=no_restart_to_current
--sufficient_reduction_for_restart=0.37 --necessary_reduction_for_restart=0.37
\end{lstlisting}

\subsubsection{Primal weight updates}
This section describes the commands corresponding to the ablation experiments in Section~\ref{sec:ablation-primal-weight} on primal weights.

The primal weight is fixed, with the bias $\texttt{\$XI} =\xi$, by appending the following arguments to the command in Listing~\ref{list:pdlp_ablation}:

\begin{lstlisting}
--primal_weight_update_smoothing=0.0
--primal_importance=$XI
\end{lstlisting}

\subsubsection{Presolve}
For the presolve ablation study in Section~\ref{sec:ablation-presolve},
the ``No presolve'' setting is evaluated by applying PDLP to the original (non-presolved) version of the \texttt{MIP Relaxations} dataset. See \texttt{benchmarking/README.md} for more information on the dataset generation.

\subsubsection{Diagonal preconditioning}
This section describes the commands corresponding to the ablation experiments in Section~\ref{sec:app-diagonal} on diagonal preconditioning.

The ``No scaling'' setting corresponds to appending the following arguments to the command in Listing~\ref{list:pdlp_ablation}:

\begin{lstlisting}
--l_inf_ruiz_iterations=0
--pock_chambolle_rescaling=false
\end{lstlisting}

The ``Ruiz'' setting corresponds to appending the following argument to the command in Listing~\ref{list:pdlp_ablation}:

\begin{lstlisting}
--pock_chambolle_rescaling=false
\end{lstlisting}

The ``Pock-Chambolle'' setting corresponds to appending the following argument to the command in Listing~\ref{list:pdlp_ablation}:

\begin{lstlisting}
--l_inf_ruiz_iterations=0
\end{lstlisting}

The ``Ruiz + Pock-Chambolle'' setting is PDLP.
\end{document}